\documentclass[11pt,a4paper,fleqn]{article}
\pdfoutput=1

\usepackage[english]{babel}
\usepackage{a4wide}
\RequirePackage{ucs}
\RequirePackage[utf8x]{inputenc}
\RequirePackage[T1]{fontenc}

\RequirePackage{amsthm,amsmath}
\usepackage{multirow}
\usepackage{tikz}
\usetikzlibrary{decorations.pathmorphing}
\usepackage{bigdelim}
\RequirePackage{amssymb,amstext}
\RequirePackage{amsfonts}
\RequirePackage{amsbsy}
\RequirePackage{mathrsfs}
\RequirePackage{bm}
\RequirePackage{bbm}
\usepackage{multirow,bigdelim}
\usepackage{MnSymbol}
\usepackage{enumerate}
\usepackage{graphicx}
\usepackage[round]{natbib}
\usepackage{amsthm}
\usepackage{theoremref}

\usetikzlibrary{positioning,arrows}
\tikzset{
  state/.style={circle,draw,minimum size=6ex},
  arrow/.style={-latex, shorten >=1ex, shorten <=1ex}}

\newcommand{\E}{\mathbb{E}}
\newcommand{\R}{\mathbb{R}}
\newcommand{\cP}{\mathcal{P}}
\newcommand{\N}{\mathbb{N}}

\newcommand{\F}{\mathcal{F}}
\newcommand{\G}{\mathcal{G}}

\newcommand{\A}{\mathcal{A}}
\newcommand{\bP}{\mathbb{P}}

\newcommand{\cL}{\mathcal{L}}

\newtheorem{theorem}{Theorem}[section]
\newtheorem{definition}{Definition}
\newtheorem{proposition}[theorem]{Proposition}
\newtheorem{lemma}[theorem]{Lemma}
\newtheorem{corollary}[theorem]{Corollary}
\newtheorem{remark}{Remark}

\title{Bayesian Nonparametric Inference for M/G/1 Queueing Systems}

\usepackage{authblk}

\author[1]{Moritz von Rohrscheidt\thanks{rohrscheidt@uni-heidelberg.de}}
\author[2]{Cornelia Wichelhaus \thanks{wichelhaus@mathematik.tu-darmstadt.de}}

\affil[1]{Ruprecht-Karls Universit{\"a}t Heidelberg}
\affil[2]{Technische Universit{\"a}t Darmstadt}

\begin{document}

\maketitle

\begin{abstract}
In this work, nonparametric statistical inference is provided for the continuous-time M/G/1 queueing model from a Bayesian point of view. The inference is based on observations of the inter-arrival and service times. Beside other characteristics of the system, particular interest is in the waiting time distribution which is not accessible in closed form. Thus, we use an indirect statistical approach by exploiting the Pollaczek-Khinchine transform formula for the Laplace transform of the waiting time distribution.   Due to this, an estimator is defined and its frequentist validation in terms of posterior consistency and posterior normality is studied. It will turn out that we can hereby make inference for the observables separately and compose the results subsequently  by suitable techniques.
\end{abstract}


\section{Introduction}

Much attention has been drawn to the analysis and the statistics of queueing systems. Beside approaches to statistical inference for queueing systems in the classical frequentist sense, during the last three decades there has taken place mentionable research in what is frequently called Bayesian queues. By that we mean inference for queueing systems from the Bayesian point of view. Bayesian statistical approaches are often feasible and useful in the area of Operations Research since one is able to express prior knowledge about the system which is often present, see also e.g. the foreword by Dennis Lindley in \citet{de1974theory}. It is worth to mention here some of the recent works on Bayesian queues. \citet{armero1994prior} deals with Markovian queues in the Bayesian sense, i.e. queues with a homogeneous Poisson process as input stream and exponential service times with varying number of service stations. In this parametric model, the parametrically conjugate Gamma prior is put on the parameter space which contains the arrival and service parameter with the aim to evolve a Bayesian test for the stability of the  queueing system.\\
This work forms the base for further works on Bayesian inference for systems with Markovian characters under several generalizations cf. e.g. \citet{armero1998inference, armero2004statistical, armero2006bayesian} and \citet{ausin2007bayesian, ausin2008bayesian}.\\
Generalizations in many ways were performed as for example a non-Markovian setting for the inter-arrival times, cf.  \citet{wiper1998bayesian} and \citet{ausin2007bayesian, ausin2008bayesian}.
Assuming the arrival stream of customers to be Poisson is often well-suited and ensures desirable properties of the stochastic process in the background of the system, meaning that there is a description in terms of an embedded Markov chain. In contrast, the service times often have to be modeled to be governed by a more general distribution. Moreover, more flexible classes of service time distributions often preserve the aforementioned property. Thus, it is reasonable to generalize the distribution of the service times which, in the classical parametric treatment, is taken to be exponential. Models with generalizations of the service time distribution can be found in \citet{insua1998bayesian}, where Erlang and hyperexponential distributions are employed to model the service times more flexibly. 
\citet{ausin2004bayesian} use the well known probabilistic result that the class of phase-type distributions forms a dense set in the space of all probability distributions  $[$cf.~\citet[p.84]{asmussen2008applied}$]$. They model the service time distribution semi-parametrically, assign prior distributions to its parameters for eventually inferring the system by MCMC procedures.\\
However, one is often interested in a even  more general approach for modeling the service time distribution which leads to a non-parametric Bayesian approach to estimate the unknown service time distribution.
The first work in a discrete-time framework concerning this was given by \citet{conti1999larges} for a $Geo/G/1$ system which can be seen to be the discrete-time analog of the continuous-time $M/G/1$.  $Geo/G/1$ is frequently used to model communication systems where information encoded in packages of fixed size is transmitted by a serving unit. The server is assumed to be able to transmit one package during one time slot. Hence, the randomness of the service times can rather be seen to be given in form of random batch sizes, i.e. marks of the point process governing the arrival stream. Employing a Dirichlet process prior for the distribution of the magnitude of these marks, Conti obtains estimators for various characteristics of the queue by a well known functional relationship $[$see also e.g. \citet{grubel1992functional}$]$ of the inter-arrival and the service time distribution with the waiting time distribution. Since the target of Conti's work was to obtain estimates for the waiting times of customers, he took this rather indirect approach mainly because assigning a prior distribution to the waiting time distribution and updating it by data consisting of the marked arrival stream is infeasible.
He obtained large sample properties for the estimator of the probability generating function of the waiting time distribution. These were uniform posterior consistency and a Bernstein-von Mises type result. The former roughly says that the posterior distribution, i.e.~the prior updated by the data, will center around the true value while the latter forms a Bayesian analog to the central limit theorem and gives an idea how this centering looks like. For the Bernstein-von Mises theorem results in \citet{freedman1963asymptotic} were employed.\\

The aim of the present work is to take a similar way of Bayesian statistical inference for the continuous-time $M/G/1$ system which is the more appropriate model in many situations where time-continuity is more reasonable. For example one could think of customers arriving at a cashier in a supermarket, cars arriving at a traffic jam, goods arriving at a storing center and many others. Therefor, a continuous-time analog for the functional relationship of observables and objects of interest is used which is known as the functional Pollaczek-Khinchine formula in honor of Felix Pollaczek and Aleksandr Khinchine who derived the steady state behaviour of the $M/G/1$ system in the 1930s. Thereby, the philosophy is analog to that of \citet{conti1999larges}, meaning that the main target will be the nonparametric estimation of the waiting time distribution on basis of observations of the arrival stream and the service times. However, a richer class of prior distributions is used which allows to express prior knowledge more flexibly. We feel that this is more reasonable for a system evolving in continuous time. The combination of this larger class of prior distributions together with the assumption of continuous time leads to more intricate proofs of results related to the discrete-time analog. As usual, the Bayesian approach is appreciable if prior knowledge and only few data is available. However, since typically large data samples are accessible, large sample properties as posterior consistency and posterior normality are investigated.\\

The paper is organized the following way. In section 2 the underlying queueing model is introduced and the functional relationship between observables and characteristics of interest is briefly described. Section 3 is devoted to the assignment of prior distributions to the random distributions governing the observables. Therefor, we review some facts of Bayesian statistics and call to mind a large family of non-parametric priors for eventually obtaining suitable estimators. Subsequently, frequentist validations of the suggested estimators is given in section 4 and section 5. While section 4 deals with the concentration of the posterior law of the random quantities in form of posterior consistency results, section 5 is devoted to the question how these concentrations take place  yielding Bernstein-von Mises type results.

\section{The Queueing Model}

The present section is devoted to the underlying queueing model.
The model is chosen, in Kendall's notation $[$\citet{kendall1953stochastic}$]$, to be the $M/G/1/\infty$-FIFO system. This notation means that (\textbf{M}) indistinguishable customers arrive consecutively according to a homogeneous Poisson process. Here, the \textbf{M} stands for "Markovian" to depict the memorylessness property of the exponential distribution which governs the independent inter-arrival times $(A_n)_{n \in \mathbb{Z}}$. Note that most of the appreciable properties of the $M/G/1$ system are due to this assumption.
Moreover, customers requiring service are served according to a general service time distribution (\textbf{G}) concentrated on $\R_+$. The service times are assumed to form a sequence of i.i.d.~random variables $(S_n)_{n \in \mathbb{Z}}$. The consecutive services are accomplished  by one (\textbf{1}) reliable service station in a first-in-first-out (\textbf{FIFO}) manner.
 Customers who cannot be served immediately are stored in an infinitely large waiting-room (\textbf{$\infty$}) and thus form a queue which, provided the queue is stable and in steady state, is distributed according to a stationary distribution.
For the sake of clarity of what follows, we introduce the following notation. We hereby assume that all random variables are defined on a common underlying probability space $(\Omega, \A, \bP).$

\begin{itemize}
\item $Q_t$ $\equiv$ queue length at time $t \in \mathbb{R}$
\item $N_t$ $\equiv$ number of customers in the entire system at time $t \in \mathbb{R},$ \\ i.e.~$N_t = Q_t +1$ in case $Q_t > 0$ for $t \in \mathbb{R}$
\item $W_n$ $\equiv$ waiting time of customer $n\in \mathbb{Z}$
\item $D_n$ $\equiv$ sojourn time in the system of customer  $n\in \mathbb{Z}$
\item $A_n$ $\equiv$ time between arrival of $n$-th and $(n-1)$-st customer $n\in \mathbb{Z}$
\item $S_n$ $\equiv$ service time of customer $n\in \mathbb{Z}$
\end{itemize}
Letting $\lambda$ denote the arrival rate of the Poisson arrival stream, $\mu$ the mean of the amount of time the server needs to complete the service of a customer and defining the so-called traffic intensity $\rho:=\E[S]/\E[A]=\lambda \mu$, a well known assumption for the existence of a stationary distribution for $N$ is $\rho<1$. If the queue is stable, by stationarity it is common to denote above quantities without an index tacitly assuming that the system has already run for an infinitely large time-horizon. If additional one assumes $\E[S^2]<\infty$ it can be shown that
\begin{align*}
\E[N]=\rho+\frac{\rho^2+\lambda\textsf{Var}[S]}{2(1-\rho)}.
\end{align*}
This formula is known as the Pollaczek-Khinchin mean formula and relates the mean of the distribution of $N$ to that of $A$, $S$ and to the variance of $S$, $\textsf{Var}[S]$. However, this formula only relates parameters of aforementioned distributions to each other instead of the entire distributions. However, in Bayesian statistics one is typically interested in the entire distribution of a parameter one is uncertain in. A functional relationship is given by the Pollaczek-Khinchine transform formula. To state it, define the following objects. Let $n(z)=\sum_{k=0}^{\infty} z^k \bP(N=k)$ be the probability generating function (p.g.f.) of the distribution of $N$ and
$g(z)=\int_0^{\infty}e^{-zs} dG(s)$ the Laplace-Stieltjes transform (LST) of the service time distribution. Then the following functional relationship can be obtained, see e.g. \citet{haigh2004prob} or \citet{nelson2013probability}.
\begin{align*}
n(z)=g(\lambda(1-z))\frac{(1-z)(1-\rho)}{g(\lambda(1-z))-z}, \qquad z\in[0,1].
\end{align*}
Manipulations of this formula give
\begin{align*}
&q(z)=\frac{(1-z)(1-\rho)}{g(\lambda(1-z))-z}, \qquad z\in[0,1],  \qquad \text{ and }
&w(z)=\frac{z(1-\rho)}{z-\lambda(1-g(z))}  , \qquad z\in \R_+,
\end{align*}
where $q(z)$ denotes the p.g.f.~of the queue-length distribution and $w(z)$ the LST of the waiting-time distribution. The latter two quantities are of special statistical interest since they provide the essential information about the development of the queue.

\section{Prior Assignments and Estimators}

All the transforms $n(z)$, $q(z)$ and $w(z)$ introduced in section 2 depend on the arrival rate $\lambda$ and on the LST of the service time distribution $g(z)$. These values are typically assumed to be unknown and thus need to be inferred. For this inference, we choose a Bayesian approach which is further introduced in the present section. Since it does not make sense to a subjectivist statistician to talk about "fixed but unknown" parameters $[$see \citet{de1974theory}$]$, \\  another philosophical concept is used. That is one interprets the "fixed but unknown" parameter as a random quantity itself. This approach leads to the concept of exchangeability which provides a "meaningful, observable character" of the data.  This approach is briefly reviewed. Let $(\Omega, \A, \bP)$ be an abstract probability space.
Call an infinite sequence of random variables $(X_i)_{i=1}^{\infty}$  with $X_i:\Omega\rightarrow \R, i \in \mathbb{N},$  exchangeable if for any $m\in\mathbb{N}$ and any permutation $\pi$ of $m$ elements it holds that
\begin{align}
\mathcal{L}\left[X_1,\dotsc,X_m\right]=\mathcal{L}\left[X_{\pi(1)},\dotsc,X_{\pi(m)}\right], \tag{E}
\end{align}
where $\mathcal{L}$ denotes the joint law of the respective random object.\\

Now, assume that we observe $n$ inter-arrival times $A_1^n=(A_1,\dotsc,A_n)$ between the first $(n+1)$ consecutive customers as well as the service times $S_1^n=(S_1,\dotsc,S_n)$ of the first $n$ customers. The data $A_1^n$ and $S_1^n$ are assumed to be the first $n$ projections of two independent sequences of exchangeable random variables $A_1^{\infty}$ and $S_1^{\infty}$ which, in turn, are assumed to be independent of each other, i.e. $\mathcal{L}(S_i;i\in I|A_j;j\in J)=\mathcal{L}(S_i;i\in I)$ for all finite subsets $I,J\subset \N$.\\

We turn to de Finetti's theorem for Polish spaces $[$\citet{hewitt1955symmetric}$].$  Let $\cP(\mathcal{S})$ denote the space of all probability measures on some Polish space $\mathcal{S}$ and consider in particular
$\cP(\R)$ equipped with the topology of weak convergence of measures. This leads to a measurable space $(\cP(\R),\mathfrak{B}_{\cP(\R)})$ which is itself Polish $[$e.g.~\citet{kechris1995descriptive}$]$. By the de Finetti theorem for exchangeables, it holds for all $n\in \N$ and for all measurable subsets $A_i\subset \R$, $i=1,\dotsc,n$ that $X_1^{\infty}=(X_1,X_2\dotsc)$ is exchangeable if and only if there is a unique mixing measure $\nu\in\cP(\cP(\R))$ such that
\begin{align*}
\bP(X_i\in A_i ; i=1,\dotsc,n)=\int\limits_{\cP(\R)} \prod\limits_{i=1}^n P(X_i\in A_i) \nu(dP).
\end{align*}

The right-hand side reflects the equivalent property of the sequence $X_1^{\infty}$ being exchangeable the following way.  The data $X_1^{\infty}$ are conditionally i.i.d.~given some probability measure $P\in\cP(\R)$, in symbols write $X_1^{\infty}|P \sim \bigotimes_{\mathbb{N}} P$. The probability measure $P$ itself is random and distributed according to the mixing measure $\nu$ which is called the prior distribution from a Bayesian statistical point of view. Moreover, notice that unconditionally the data $X_1^{\infty}$ are in general not independent. Indeed, by a result from \citet{kingman1978uses}, it is easily seen that in general  data are positively correlated, a fact that makes Bayesian statistics a theory of statistical prediction and thus convenient for other theories as e.g. machine learning.

For actual applications, however, the prior $\nu$ and the integration in de Finetti's theorem becomes infeasible in general. This is mainly due to the fact that  one has merely a non-constructive proof of the existence of $\nu$. Another difficulty is that the set $\cP(\R)$ is quite large. One has mainly two ways to choose to circumvent this problem in applications. The first is to shrink the support of $\nu$ to a reasonable subset of $\cP(\R)$ by an additional "symmetry constraint" on the sequence of random variables, thus, to put  additional information on the data, to boil down above integration to a finite dimensional parameter space. The second is to model $\nu$ non-parametrically by genius probabilistic tools. We use the first approach here  to get a prior distribution for the random arrival rate $\lambda: \Omega \rightarrow \R_+$ and the second to get a prior distribution for the random service time distribution $G:\Omega \rightarrow \cP(\R_+)$.

\subsection{Arrival Rate}
 In case of the inter-arrivals an additional symmetry assumption on the law of the infinite exchangeable sequence $A_1^{\infty}$ will give rise to a mixing measure which is supported by all exponential distributions. This additional symmetry condition is stated as follows. Let $n\in \N$ and for $i=1,...,n$, $B_i \in \mathfrak{B}_{\R_+}$, the Borel sigma-field on $\R_+$. Furthermore, let $c_i\in \R$ be real constants such that $\sum_{i=1}^n c_i=0$ and $C_i=c_i+B_i=\{c_i+r : r\in B_i\} \subset \R_+$, $i=1, \dots, n.$ The law of the exchangeable sequence $A_1^{\infty}$ fulfills the symmetry condition
\begin{align}
\bP(A_i\in B_i; i=1,...,n)=\bP(A_i\in C_i; i=1,...,n), \tag{S}
\end{align}
for all $n,B_i,C_i$ as above if and only if it is a mixture of exponential distributions $[$\citet{diaconis1985quantifying}$]$, in symbols
\begin{align*}
\bP(A_i\in B_i;i=1,...,n)&=\int\limits_{\mathcal{E}} \prod\limits_{i=1}^n P(A_i\in B_i) \nu(dP)\\
&=\int\limits_{\R_+} \prod\limits_{i=1}^n \int_{B_i} \lambda e^{-\lambda x_i} dx_i \tilde{\nu}(d\lambda),
\end{align*}
where $\mathcal{E}$ denotes the space of all exponential distributions on $\mathbb{R}_+$ and $\tilde{\nu}$ denotes the push-forward measure of $\nu$ along the natural parametrization $\tilde{} : \mathcal{E}\rightarrow \R_+;
 P\mapsto \lambda$
of the exponential distributions.
Moreover, if in above situation additionally it holds $\E[A_2|A_1]=\alpha A_1+\beta$ for real constants $\alpha,\beta>0$, then the mixing measure $\tilde{\nu}$ can be shown to be a Gamma distribution.\\

Since the $M/G/1$ queueing model implies a mixture of exponential distributions for the joint law of the inter-arrivals, from a Bayesian point of view, we assume $A_1^{\infty}$ to meet constraints (E), (S) and that $\E[A_2|A_1]=\alpha A_1+\beta$ holds as well (note that by exchangeability, this extends to all random inter-arrival times). The latter directly leads to a conjugate prior for the arrival rate. To be more precise, if we assume $\lambda$ to be a random variable distributed according to a Gamma distribution with hyper parameters $a,b>0$, then the posterior distribution given the data $X_1^n$  is a Gamma distribution as well with updated hyper-parameters $(a+n,b+\sum_{i=1}^{n}A_i)$, i.e.~the family of gamma distributions is closed with respect to exponential sampling.
In summary, we assume the following sampling scheme
\begin{align*}
\lambda | (a,b) &\sim \Gamma(a,b)\\
A_1^{\infty}|\lambda &\sim \bigotimes\limits_{\N}\mathcal{E}(\lambda),
\end{align*}
which leads, through Bayes theorem, to the posterior distribution, the Bayes estimate for squared error loss and the posterior predictive distribution density, respectively, given by
\begin{align*}
\lambda | (a,b),A_1^n &\sim \Gamma\left(a+n,b+\sum_{i=1}^n A_i\right)\\
\E_{\Gamma}[\lambda | A_1^n,(a,b)]&= \frac{a+n}{b+\sum_{i=1}^n A_i}\\
f\left(a_{n+1}|A_1^n,(a,b)\right)&=\frac{(a+n)\left(b+\sum_{i=1}^n A_i\right)^{a+n}}{\left(b+\sum_{i=1}^n A_i+a_{n+1}\right)^{a+n+1}}.
\end{align*}
The latter leads to the predictive value for the next observation
\begin{align*}
\E[A_{n+1}| A_1^n, (a,b)]&=\frac{1}{a+n-1}\sum_{i=1}^n A_i +\frac{b}{a+n-1}.
\end{align*}
Note that the latter equation again reflects the learning process which does not exist in the frequentistic approach in such an explicit form and that for $n=1$ it is given by $\E[A_2|A_1,(a,b)]=1/a A_1 + b/a$.

\subsection{Service Time Distribution}
Since the $M/G/1$ model does not imply a parametric mixture for the service time random variables as for the inter-arrivals, assigning a suitable prior is a more difficult task. Not being able to shrink the support of the mixing measure $\nu$ in the de Finetti theorem to a finite-dimensional set, we have to choose a prior that supports most of $\cP(\R_+)$. The common way is to parametrize $\cP(\R_+)$ by a reasonable dense subset. This is taken to be the set of all discrete distributions on $\R_+$, $\cP_d(\R_+)=\{P\in \cP(\R_+): P(\cdot)=\sum_{i=1}^{\infty} w_i \delta_{x_i}(\cdot); w_i \in[0,1], \sum_{i=1}^{\infty}w_i=1, x_i\in \R_+; \forall i\in\N\} $.
 It is well-known that the most famous non-parametric prior in Bayesian statistics, namely the Dirichlet process prior obtained in \citet{ferguson1973bayesian}, samples discrete probability measures with probability one. Moreover, it is known that it has full weak support of all of $\cP(\R_+)$ if its base measure has support all of $\R_+$. However, here we will use a slightly more general family of prior distributions that enables us to model more flexibly prior beliefs of the true data generating distribution $G_0$.
This larger family will be a subclass of so called neutral to the right prior processes, namely the beta-Stacy processes. Although these priors sample discrete probabilities with probability one, too, an analogue of the result concerning the weak support is known as well. The class of neutral to the right priors is now briefly introduced and its most important properties will be stated.

\subsubsection{Neutral to the right Priors}
Let $\mathcal{F}(\R_+)$ denote the space of all cumulative distribution functions (c.d.f.) on $\R_+$. Then, a random distribution function $F\in \Omega^{\mathcal{F}(\R_+)}$ is said to be neutral to the right (NTR) if for each $k>1$ and 0<$t_1<t_2<...<t_k$ the normalized increments
\begin{align*}
F(t_1),\frac{F(t_2)-F(t_1)}{1-F(t_1)},...,\frac{F(t_k)-F(t_{k-1})}{1-F(t_{k-1})}
\end{align*}
are independent assuming $F(t)<1, \forall t\in\R_+$. That is, for all $i\in\N$, $\bar{F}(t_i)/\bar{F}(t_{i-1})$ is independent of the sigma-field generated by $F$ up to time $t_{i-1}$, $\sigma(\{F(t):t<t_{i-1})$, where $\bar{F}(\cdot):=1-F(\cdot)$ is the survival function associated to $F$. This essentially asserts that the proportion of mass that $F$ assigns to $(t_i,\infty)$ with respect to $(t_{i-1},\infty)$  does not depend on how $F$ behaves left of $t_{i-1}$.
 This property coined the name neutral to the right. \citet{doksum1974tailfree} has shown that $F(\cdot)\in \Omega^{\mathcal{F}}$ is NTR if and only if $\mathcal{L}(F(\cdot))=\mathcal{L}(1-\exp[-A(\cdot)])$ for some independent increment process $A(\cdot)$ which is almost surely non-decreasing, right continuous and such that $\lim_{t\rightarrow -\infty}  A(t)=0$ and $\lim_{t\rightarrow \infty} A(t)=\infty$. Such objects  are called increasing additive processes, see e.g. \citet{sato1999levy}. For more details on the construction of NTR priors see e.g. \citet{phadia2015prior}.
  Since independent increment processes are well understood, the definition of NTR priors leads to a rich class of non-parametric priors which are analytically tractable.
Another nice feature of NTR priors is that this family is conjugate with respect to (right-censored) exchangeable data. A fact that makes NTR priors appreciable in statistical survival analysis. Notice, that a Dirichlet process prior updated by right censored data is not longer Dirichlet.\\

The next proposition makes a statement about the weak support of a NTR prior. Recall that the topological support of a measure is the smallest closed set with full measure. In the following theorem $\mathcal{F}(\R_+)$ is identified with the space of all probability measures $\cP(\R_+)$ which is equipped with the weak topology.

\begin{proposition}\citet{dey2003some}\\
Let $F$ be a random distribution function which is governed by a NTR prior $\Pi\in\cP(\mathcal{F})$, i.e. $F\sim \Pi$, and let $A(\cdot)=-\log(1-F(\cdot))$ be the corresponding positive increasing additive process with Lévy measure $L$. Then, $\Pi$ has full support if $L$ has full support. The assertion remains true if $\mathcal{F}(\R_+)$ is equipped with the sup-norm.
\end{proposition}
\begin{proof}
It has to be shown that all weak neighborhoods of any probability measure have positive $\Pi$-mass. Since continuous distributions are dense in $\cP(\R_+)$ with respect to the weak topology, it suffices to show the assertion for all weak neighborhoods of continuous distributions. Now, choose some random distribution function $F_0$ and some $\epsilon > 0$ and consider $U_{\epsilon}:=\left\{F:\sup\limits_{0\leq t <\infty} \left|F(t)-F_0(t)\right|<\epsilon\right\}$. Then $[$e.g. \citet{ghosh2003bayesnonp})$]$ $W\subseteq U_{\epsilon}$ for some weak neighborhood $W$ of $F_0$. Moreover, there is $\delta > 0$, $m\in\N$ and $0 < t_1 < t_2 < \cdots < t_m$ such that $\left\{F:\left|F(t_i,t_{i+1}]-F_0(t_i,t_{i+1}]\right|<\delta, i=1,\dotsc,m \right\}\subseteq W$. Hence, it is enough to show that sup-neighborhoods restricted to compact sets possess positive prior probability. By the homeomorphism $\phi:F(\cdot)\mapsto -\log[1-F(\cdot)]$ the problem can be translated to the analogous one for the corresponding Lévy process. That is to show that $\mathcal{L}(A)$ gives positive probability to sets of the form $C:=\left\{A:\sup\limits_{0\leq t \leq r}\left|A(t)-\phi^{-1}F_0(t)\right|<\gamma \right\}$ for fixed $r\in \mathbb{Q}$. Now, take a partition $\rho=\cup_{i=1}^k (a_i,a_{i+1}]$ of $(0,r]$  such that $\sup\limits_{1\leq j\leq k}\phi^{-1}F_0(a_j,a_{j+1}]<\gamma$ and define
\begin{align*}
&B_i:=(a_i,a_{i+1}]\times (\phi^{-1}F_0(a_i,a_{i+1}]-\gamma/k,\phi^{-1}F_0(a_i,a_{i+1}]+\gamma/k), \\
&B:=\cap_{i=1}^k\left\{A: \#\{(t,A(\{ t \}) \in B_i\}=1 \right\}.
\end{align*}
It follows that $B\subseteq C$ and, since $L$ was assumed to have full support,
\begin{align*}
\phi^{-1}\Pi(C)\geq \phi^{-1}\Pi(B)=\prod\limits_{i=1}^k L(B_i)e^{-L(B_i)}>0.
\end{align*}

\end{proof}

\subsubsection{Beta Processes and Beta-Stacy Processes}

For our purposes, we choose a NTR prior with corresponding process $Y(\cdot)$ being driven by a certain class of Lévy measures.
These were studied by \citet{hjort1990nonparametric} and \citet{walker1997beta}. Hjort studied beta-processes from a survival analysis viewpoint and therefor elicited a non-parametric prior for the cumulative hazard function (c.h.f.) given for $F\in \mathcal{F}(\R_+)$ by
\begin{align*}
H(t)=\int_0^t \frac{dF(s)}{\bar{F}(s)}, \: t >0.
\end{align*}
\citet{walker1997beta} give the definition of the analog of the beta-process as a prior for the c.d.f. directly as follows.
$F$ is said to be distributed according to a beta-Stacy process with parameters $(c(\cdot),H(\cdot))\in\R_+^{\R_+}\times \mathcal{F}(\R_+)$ (for short $F\sim BS(c,H)$) if for all $t\geq 0$ the corresponding process $Y(\cdot)$ fulfilling $F(\cdot)=1-\exp[-Y(\cdot)]$ has Lévy measure
\begin{align*}
dL_t(x)=\frac{dx}{1-e^{-x}}\int\limits_0^t e^{-xc(s)(1-H(s))} c(s) dH_c(s),
\end{align*}
for all $x>0$, where $H_c(t)=H(t)-\sum\limits_{k:t_k<t}H({t_k})$ is the continuous part of $H$ with  $t_i$ as the fixed points of discontinuity of $H$. Since we are in a continuous-time framework, we will always choose $H$ to be continuous. However, discontinuities appear in the Lévy measure governing the posterior law. Note that $\E_{BS}[F(\cdot)]=H(\cdot)$ is the prior guess on the c.d.f.~$F$ and the function $c(\cdot)$ acts like a tuning parameter affecting the magnitude of the increments and is sometimes interpreted as the "flexible" belief in the prior guess. For the sake of clarity, note that the Dirichlet process with finite measure $\alpha$ as parameter admits a similar representation with Lévy-measure
\begin{align*}
dD_t(x)=\frac{dx}{1-e^{-x}}\int\limits_0^t e^{-xc (1-\bar{\alpha}(s))} c \bar{\alpha}(ds),
\end{align*}
where $c=\alpha(\R)$ and $\bar{\alpha}(\cdot)$ denotes the c.d.f. corresponding to the probability measure $\frac{\alpha(\cdot)}{c}$. In the case of Dirichlet processes it is well-known that the prior guess on the random probability measure equals $1/c\alpha(\cdot)$ and $c$ itself is often interpreted as the strength of belief in the prior guess. Thus, a Dirichlet process is a beta-Stacy process whose parameters are determined by $\alpha$ alone.

As already mentioned, the beta-Stacy process is a parametrically conjugate prior, meaning that the posterior law of $F$ given exchangeable possibly right-censored data is a beta-Stacy process as well.
 To be more precise, let $S_1^n$ be the first $n$ projections of infinite exchangeable data $S_1^{\infty}$. Let $S_1^{\infty}$ be conditional iid given the c.d.f. $F$ and let $F\sim BS(c,H)$.

Further, define $M(t)=\sum\limits_{i=1}^n 1_{[t,\infty)}(S_i)$ and $N(t)=\sum\limits_{i=1}^n 1_{[0,t]}(S_i)$. Then it holds $[$\citet[Theorem 4]{walker1997beta}$]$ that $F|S_1^n \sim BS(c_n^*,H_n^*)$, where $c_n^*(\cdot)$ and $H_n^*(\cdot)$ are given by
\begin{align*}
H_n^*(t)&=1-\prod\limits_{s \in [0,t]}\left(1-\frac{c(s)dH(s)+dN(s)}{c(s)\bar{H}(s)+M(s)}\right),\\
c_n^*(t)&=\frac{c(t)\bar{H}(t)+M(t)-N(t)}{\bar{H^*}(t)}.
\end{align*}
Thereby, $H_n^*(t)$ is defined by means of the product integral, see \citet{gill1990survey}.
Note, that the posterior process possesses fixed points of discontinuity at the observations and that the posterior guess on $F$, i.e.~the Bayes estimate with respect to squared error loss is given by $\hat{F}_n(\cdot)=H_n^*(\cdot)$.\\

Having a look at the Pollazcek-Khinchine transform formulas in section 2, one may ask if they are well defined for almost all $\lambda$ and $F$ drawn from their respective prior and posterior distributions.
While the posterior of the mean inter-arrival times is straightforward by definition, the mean and the second moment of the random c.d.f. describing the general service times deserves more attention.  The existence of functionals of c.d.f.s drawn according to a beta-Stacy process was studied in \citet{epifani2003exponential}.  Relating $A(\cdot)$ and $c(\cdot)$ to the existence of a certain functional, they obtained sufficient conditions for moments of order $m$ to exist $[$equation (10) in their article$]$.  We will assume throughout that this condition holds at least for the moment of second order. Moreover, note that they obtained an explicit formula for the prior moments $[$equation (11)$]$ as well as for the posterior moments $[$equation (13)$]$.

In summary, we assume the following sample scheme for the service times in the $M/G/1$ system.
\begin{align*}
G|(c,H) &\sim BS(c,H)\\
S_1^{\infty}|G &\sim \bigotimes_{\N}G.
\end{align*}
This leads to the posterior distribution
\begin{align*}
G|(c,H),S_1^n \sim BS(c_n^*,H_n^*).
\end{align*}

Notice again that this updating continues to hold for right-censored observations. Hence, the model can be enlarged in that one does not have to keep exact track of the customers and is still able to make reasonable inference even if e.g. it is solely known that the service of customers has exceeded a certain threshold.

\subsection{Estimators for Queueing Characteristics}

In this subsection, we study the Bayes estimators, i.e. the posterior means, of several characteristics of the $M/G/1$ system. Not all posterior laws of each single characteristic are obtainable explicitly. Hence, for those characteristics whose posterior laws are not obtainable in closed form, we define suitable estimators by replacing Bayes estimators for the corresponding values. \\

Assume for the general service time distribution $G\sim BS(c,H)$ and let
\begin{align*}
&\hat{G}_n=\E_{BS}\left[G|S_1^n\right],\\
&\hat{\mu}_n=\E_{BS}\left[\int_0^{\infty} t dG(t)|S_1^n\right] \text{ and} \\
&\hat{\lambda}_n=\E_{\Gamma}\left[\lambda| A_1^n\right]
\end{align*}
 be the Bayes estimates with respect to squared error loss. Note that with $M_n(s) = \sum_{j=1}^n \delta_{X_j} [s, \infty)$, as in \citet{epifani2003exponential},  $\hat{\mu}_n$ can be given    as
\begin{align*}
\hat{\mu}_n=\int_0^{\infty} \exp\left[-\int_0^t \frac{\alpha(ds)}{\beta(s)+M_n(s)}\right] \exp\left[-\int_0^t \frac{\beta(s)+M_n(s)-1}{\beta(s)+M_n(s)}N_n(ds)\right] dt,
\end{align*}
where if $H$ is continuous $[$see \citet{phadia2015prior}$]$
\begin{itemize}
\item $\beta(t)=c(t)[1-H(t)]$
\item $\alpha(t)=\int_0^t c(s)dH(s)$.
\end{itemize}
Observe that $\hat{\rho}_n=\E_{BS\otimes\Gamma}\left[\lambda \mu| S_1^n,A_1^n \right]=\hat{\mu}_n \hat{\lambda}_n$  by above independence assumption. However, in Bayesian statistics, one is interested in the entire posterior law rather than merely in a certain functional. The posterior law of $\rho$ is obtainable explicitly in the following form
\begin{align*}
&P_{BS\otimes \Gamma}(\rho\leq t|A_1^n,S_1^n)=P_{BS\otimes \Gamma}(\mu \lambda \leq t|A_1^n,S_1^n)\\
&=\int_0^{\infty} P_{BS\otimes \Gamma}(\mu\leq t/\lambda|A_1^n,S_1^n,\lambda) P_{BS\otimes \Gamma}(\lambda|A_1^n,S_1^n) d\lambda\\
&=\int_0^{\infty} P_{BS}(\mu\leq t/\lambda|S_1^n,\lambda) P_{ \Gamma}(\lambda|A_1^n) d\lambda\\
&=\frac{(b+\sum_{i=1}^n A_i)^{a+n}}{\Gamma(a+n)}\int_0^{\infty} P_{BS}(\mu\leq t/\lambda|S_1^n,\lambda) \lambda^{a+n-1}e^{-\lambda(b+\sum_{i=1}^n A_i)} d\lambda.
\end{align*}
Note that $ P_{BS}(\mu\leq t/\lambda|S_1^n,\lambda)$ can be stated more explicitly in form of a density by means of Proposition 4 in \citet{regazzini2003distributional}. This explicit form is omitted here due to technical reasons and left to the interested reader. If one observes the queueing system and is, as is usually the case, not aware of the condition $\rho_0<1$, one can thus obtain the posterior probability that the system is stable $P_{BS\otimes \Gamma}(\rho<1 |A_1^n,S_1^n)$. \\

We close this section by defining estimates for:
\begin{itemize}
\item mean service time:  $\hat{\mu}_n$
\item traffic intensity:  $\hat{\rho}_n=\hat{\lambda}_n \hat{\mu}_n$
\item LST of service-time distribution: $g_n^*(z)=\int_0^{\infty} e^{-zt} d\hat{G}_n(t)$
\item LST of waiting-time distribution: $w_n^*(z)=\frac{z(1-\hat{\rho}_n)}{z-\hat{\lambda}_n(1-g_n^*(z))}$
\item pgf of queue-size distribution: $q_n^*(z)=\frac{(1-\hat{\rho}_n)(1-z)}{g_n^*(\hat{\lambda}_n(1-z))-z}$
\item system-size distribution: $n_n^*(z)=q_n^*(z) g_n^*(\hat{\lambda}_n(1-z))$.
\end{itemize}

Notice that for the random arrival- and service rate and the random distribution function of the service times the natural Bayes estimator is used, i.e. the minimizer with respect to squared error loss. They are analytically tractable and obtainable in closed form. For the remaining queueing characteristics obvious plug-in estimators are used. The reason therefor is that a closed form of the push-forward law under the mapping
\begin{align*}
(\lambda,G(\cdot))\mapsto f(\cdot),
\end{align*}
where $f\in\{n,g,q\}$, of the prior $\Pi_{BS\otimes \Gamma}$ is not easy to obtain. Neither is known how to update such a pushed prior by exchangeable data $(S,A)_1^n$ which we have access to in a natural Bayesian way. \\

However, the goodness of these estimates in a rigorous mathematical sense is established in the next sections.

\section{Consistency Results}
Posterior consistency provides a tool for validation of Bayesian procedures. Roughly speaking, it is defined to be the property of the posterior law to center around the true parameter when the number of observed data increases. As a consequence of posterior consistency two different priors will asymptotically lead to the same prediction, a fact often called merging of prior opinions. So, prior information is consecutively washed away by the new state of information provided by the data. Posterior consistency is the most desired property of Bayesian procedures since it states that one can recover the true measure that generates the data. For examples of inconsistent Bayes procedures see e.g. \citet{diaconis1986consistency} and \citet{kim2001posterior} and references therein. \\

We will have to deal with two different posterior consistency issues, a parametric one for the posterior of the arrival rate and a non-parametric one for the service time distribution and values depending on it. The first was considered by \citet{doob1949application} in a rather general way by the use of martingale theory. It clarifies that, under very weak constraints (the state- and the parameter spaces are assumed to be Polish spaces and the likelihood is identifiable), there is a subset of the parameter space with full prior mass such that the sequence of posterior laws is consistent at any parameter of that subset that is taken as the true one. However, one problem is that in high-dimensional parameter spaces the set with full prior mass might become topologically small. Therefore, especially nonparametric Bayesian statistical problems deserve a deeper study of posterior consistency since the set of possible likelihoods can not longer be assumed to be parametrized finitely.\\

First of all, the definition of posterior consistency in a rather general framework is given as it can be found e.g.~in \citet{schervish1995theory}. Notice that for any Polish state space $\mathcal{S}$ the space of all probability measures, $\cP(\mathcal{S})$, can be equipped with the topology induced by weak convergence which has neighborhood bases for $P\in\cP(\mathcal{S})$ given by the collection of sets of the form $U_{P;\epsilon}=\left\{Q\in\cP(\R): \left|\int f_i dP -\int f_i dQ\right| <\epsilon, f_i\in C_b(\R), \text{ for all } i=1,..,k \right\}$. This topology makes $\cP(\mathcal{S})$ itself a Polish space $[$see \citet{kechris1995descriptive}$]$ and the topology can be metrized e.g. by the Prohorov metric. The Borel $\sigma$-field, $\mathfrak{B}(\cP)$, with respect to the weak topology serves as a natural measure-theoretical structure to turn $\cP(\mathcal{S})$ into a measurable space $(\cP(\mathcal{S}),\mathfrak{B}(\cP))$. $\mathfrak{B}(\cP)$ turns out to be the smallest $\sigma$-field on $\cP(\mathcal{S})$ that makes the mappings $P\mapsto P(A)$ $(\mathfrak{B}(\cP),\mathfrak{B}([0,1]))$-measurable for all $A\in\mathfrak{B}(\mathcal{S})$. In the sequel $\Pi$, occasionally with an appropriate index, shall denote a prior distribution and $\Pi_n$ the corresponding posterior distribution after having seen the first $n$ projections of the exchangeable data. A distribution (or a value that parametrizes a distribution) indexed by $0$ will always denote the true data generating measure.

 \begin{definition}
Let $\Pi\in\cP(\cP(\R))$ be a prior distribution, i.e. the distribution of some random probability measure $P\in\cP(\R)$ and let $X_1^{\infty}$ be a sequence of exchangeable data which is conditionally i.i.d. given $P$. Moreover, let $P_0\in\cP(\R)$ be the true data generating distribution and $\left(\Pi_n\right)_{n\geq 0}=\left(\Pi(\cdot|X_1^n)\right)_{n\geq 0}$ be the sequence of posterior laws. Then call $(\Pi_n)_{n \geq 0}$ weakly consistent at $P_0$ if for all weak neighborhoods $U_{P_0,\epsilon}$ of $P_0$ it holds
\begin{align*}
\Pi_n(U_{P_0,\epsilon})\stackrel{n \rightarrow \infty}{\longrightarrow} 1,
\end{align*}
for $P_0^{\infty}$-almost all data sequences $X_1^{\infty}$.
\end{definition}
 Here, $P_0^{\infty}$ denotes the true joint law governing the sequence $X_1^{\infty}$.\\

Needless to say that in the case of the arrival rate $\lambda$ things become easier because, due to the additional judgment concerning the symmetry of the distribution of the exchangeable sequence of inter-arrival times, one can reduce the problem to that of parametric consistency that is rather easy to handle by conjugacy.

\begin{proposition}\thlabel{prop1}
Let $\Pi_{\Gamma}$ stand for the Gamma prior of the random arrival rate $\lambda$ and $\left(\Pi_{\Gamma;n}(\cdot)\right)_{n\geq 1}=\left(\Pi_{\Gamma}(\cdot|A_1^n)\right)_{n\geq 1}$ be the sequence of posterior laws. Then, for all $\epsilon>0$ and for $P_{\lambda_0}^{\infty}$ almost all sequences $A_1^{\infty}$ it holds
\begin{align*}
\Pi_{\Gamma;n}\left(|\lambda-\lambda_0|<\epsilon\right) \stackrel{n\rightarrow\infty}{\longrightarrow}1.
\end{align*}
\end{proposition}
\begin{proof}
Taking into account that the gamma prior is conjugate for exponentially distributed data, the posterior density can be shown to be that of a $\Gamma(a+n,b+\sum_{i=1}^n A_i)$ distribution. By well known properties of the Gamma distribution one has $\E_{\Gamma;n}[\lambda]=(a+n)(b+\sum_{i=1}^n A_i)^{-1}$ and $\textsf{Var}_{\Gamma;n}[\lambda]=(a+n)\left(b+\sum_{i=1}^n A_i\right)^{-2}$.
Applying the continuous mapping theorem the assertion follows from the strong law of large numbers and the Markov inequality.
\end{proof}

Let us now turn to the nonparametric part concerning the random c.d.f. $G$. The following two results are due to \citet{dey2003some} and give the consistency of the Bayes estimator and the random c.d.f. $G$.

\begin{lemma}[\citet{dey2003some}]\thlabel{deylemma}
Let $G_0\in \mathcal{F}(\R_+)$ be the true c.d.f. of the service times. Furthermore suppose that $G_0$ is continuous and that $G_0(t)<1$ for all $t\in\R_+$. Let $G\sim BS(c,H)$ and $\Pi_{BS;n}$ denote the posterior distribution of $G$. Then, for $G_0^{\infty}$-almost all sequences of data $S_1^{\infty}$ and all $t\in\R_+$ it holds
\begin{align*}
\E_{BS;n}\left[G(t)\right]\stackrel{n\rightarrow \infty}{\longrightarrow} G_0(t).
\end{align*}
\end{lemma}

Next, the posterior consistency of the random service time distribution $G$ induced by a beta-Stacy prior will be stated. Since the service time distribution function $G(\cdot)$ is seen to be a random function rather than a random variable, we investigate deviations in the sup-norm.

\begin{theorem}[\citet{dey2003some}]\thlabel{deytheorem}
Let $G_0\in \mathcal{F}(\R_+)$ be the true c.d.f. of the service times. Again, suppose that $G_0$ is continuous and that $G_0(t)<1$ for all $t\in\R_+$. Let $G\sim BS(c,H)$ and $\Pi_{BS;n}$ denote the posterior distribution of $G$. Then, for all $\epsilon >0$ it holds
\begin{align*}
\Pi_{BS;n}\left(\sup\limits_{0\leq t <\infty} |G(t)-G_0(t)|<\epsilon \right)\stackrel{n\rightarrow\infty}{\longrightarrow}1
\end{align*}
for $G_0^{\infty}$-almost all sequences $S_1^{\infty}$.
\end{theorem}

\begin{remark}
We stress some peculiarity of a certain class of neutral to the right measures. For a full formal treatment see \citet{dey2003some}. As already mentioned in the previous section, neighborhoods w.r.t.~the sup-norm contain some weak neighborhoods. These, in turn, are given as finite intersections of sets of the form $\left\{G: \left|G(t)-G_0(t)\right|<\gamma \right\}$. Thus, it is enough to show that the posterior mass of such sets converges to one. But since the Lévy measure that corresponds to the beta-Stacy prior is of the form $L(dt,ds)=a(t,s)ds K(dt)$ for suitable $a$ and $K,$ the convergence of the expected value of the posterior law to the true c.d.f. already ensures that the posterior variance vanishes with increasing data size.
\end{remark}

As an immediate consequence one has the uniform consistency of the Bayes estimator.

\begin{corollary}\thlabel{cor2}
Let the conditions of \thref{deylemma} be fulfilled. Then the Bayes estimate $\E_{BS;n}[G(\cdot)]$ of the service time distribution is uniformly consistent at the true continuous service time distribution $G_0(\cdot)$, that is
\begin{align*}
\sup\limits_{0\leq t < \infty} \bigg\vert\E_{BS;n}[G(t)]-G_0(t)\bigg\vert \stackrel{n\rightarrow\infty}{\longrightarrow}0,
\end{align*}
for $G_0^{\infty}$-almost all sequences $S_1^{\infty}$.
\end{corollary}
\begin{proof}
By above theorem one has
\begin{align*}
&\sup\limits_{0\leq t < \infty}  \Big| \E_{BS;n}[G(t)]-G_0(t) \Big|=\sup\limits_{0\leq t < \infty}\left|\int\limits_{\mathcal{F}} \left( G(t)-G_0(t) \right) \Pi_{BS;n}(dG)\right|\\
&\leq \int\limits_{\{G:\sup \left|G(t)-G_0(t)\right|\geq\epsilon\}}
\sup\limits_{0\leq t < \infty}\left|G(t)-G_0(t)\right|\Pi_{BS;n}(dG) \\
& \qquad \qquad \qquad \qquad  \qquad \qquad + \int\limits_{\{G:\sup \left|G(t)-G_0(t)\right|<\epsilon\}}
\sup\limits_{0\leq t < \infty}\left|G(t)-G_0(t)\right|\Pi_{BS;n}(dG) \\
&<\Pi_{BS;n}\left(\sup\limits_{0\leq t < \infty}\left|G(t)-G_0(t)\right|\geq\epsilon\right) +\epsilon \stackrel{n\rightarrow \infty}{\longrightarrow} \epsilon.
\end{align*}
Since $\epsilon>0$ can be chosen arbitrarily small, the proof is completed and the assertion follows.
\end{proof}

The next lemma establishes the uniform consistency on $\R_+$ of the service time LST in posterior law.

\begin{proposition}\thlabel{prop3}
Let $g(z)=\int e^{-sz} dG(s)$ denote the LST of the random service time distribution $G$ possessing a  beta-Stacy process prior and $g_0(z)=\int e^{-sz} dG_0(s)$ the LST of the corresponding true data generating distribution. Then, if the constraints of \thref{deytheorem} are fulfilled, it holds
\begin{align*}
\Pi_{BS;n}\left(\sup\limits_{0\leq z<\infty}|g(z)-g_0(z)|<\epsilon \right)\stackrel{n\rightarrow\infty}{\longrightarrow}1,
\end{align*}
 for all $\epsilon>0$ and for $G_0^{\infty}$-almost all data sequences $S_1^{\infty}$.
\end{proposition}
\begin{proof}
The assertion of the proposition follows from \thref{deytheorem} and the continuous mapping theorem applied to the mapping $G(\cdot)\mapsto g(\cdot)$ which is continuous w.r.t. the sup-norm. Indeed, take a $\delta>0$ and let $U_{G_0,\delta}$ be a uniform $\delta$-neighborhood of the true service time distribution and let $G\in U_{G_0,\delta}$ with corresponding LST $g$.
Then, by integration by parts of the Riemann-Stieltjes integral it holds that
\begin{align*}
&\sup\limits_{0\leq z<\infty}|g(z)-g_0(z)|=\sup\limits_{0\leq z <\infty}\left|\int_0^{\infty} e^{-sz}[G-G_0](ds)\right|\\
&=\sup\limits_{0\leq z<\infty}\bigg \vert [e^{-sz}[G(s)-G_0(s)]_{|_{s=\infty}}-[e^{-sz}[G(s)-G_0(s)]_{|_{s=0}}-\int_0^{\infty} [G(s)-G_0(s)]de^{-zs} \bigg \vert \\
&=\sup\limits_{0\leq z<\infty}\left|\int_0^{\infty}[G(s)-G_0(s)] de^{-zs}\right|=\sup\limits_{0\leq z<\infty}\left|\int_0^{\infty}[G(s)-G_0(s)] ze^{-zs} ds\right|\\
&\leq \sup\limits_{0\leq z<\infty}\int_0^{\infty}ze^{-zs} \left|G(s)-G_0(s)\right| ds\\
&<\delta \sup\limits_{0\leq z<\infty}\int_0^{\infty}ze^{-zs} ds=\delta,
\end{align*}
which completes the proof and shows the claim.
\end{proof}

Next, the consistency of the estimator $g_n^*(z)$ of the LST of the service time distribution is studied.
On basis of the previous result, the following lemma establishes the uniform consistency on $\R_+$ of the estimator $g_n^*(z)$.

\begin{lemma}\thlabel{lemma4}
Let $g_n^*(\cdot), g_0(\cdot)$ as above. Then, under the constraints of Theorem \ref{deytheorem}, it holds
\begin{align*}
\sup\limits_{0\leq z <\infty}|g_n^*(z)-g_0(z)|\stackrel{n\rightarrow\infty}{\longrightarrow} 0,
\end{align*}
for $G_0^{\infty}$-almost all sequences of data $S_1^{\infty}$.
\end{lemma}
\begin{proof}
Let $\epsilon>0$ be arbitrary and let $\hat{G}_n(\cdot):=\E_{BS;n}[G(\cdot)]$. Then by properties of the Riemann-Stieltjes integral, one has
\begin{align*}
&\sup\limits_{0\leq z <\infty} |g_n^*(z)-g_0(z)| \\
&=\sup\limits_{0\leq z <\infty} \left|\int_0^{\infty} e^{-zs} d\hat{G}_n(s)-\int_0^{\infty} e^{-zs} dG_0(s)\right| \\
&=\sup\limits_{0\leq z <\infty} \big\vert \left[e^{-zs}\hat{G}_n(s)\right]_{|_{s=\infty}}-\left[e^{-zs}\hat{G}_n(s)\right]_{|_{s=0}}-\int_0^{\infty}\hat{G}_n(s)  de^{-zs}\\
&\qquad \qquad \qquad \qquad \qquad \qquad \qquad -\left[e^{-zs}G_0(s)\right]_{|_{s=\infty}}+\left[e^{-zs}G_0(s)\right]_{|_{s=0}} +\int_0^{\infty} G_0(s) de^{-zs} \big\vert \\
&=\sup\limits_{0\leq z <\infty}\left| \int_0^{\infty}\hat{G}_n(s)d(e^{-sz})-\int_0^{\infty} G_0(s)d(e^{-sz}) \right|\\
&= \sup\limits_{0\leq z <\infty} \left|\int_0^{\infty}[\hat{G}_n(s)-G_0(s)]ze^{-zs}ds\right|\\
&< \sup\limits_{0\leq s <\infty}\left|\hat{G}_n(s)-G_0(s)\right|.
\end{align*}
Hence, the assertion follows from \thref{cor2} and the proof is completed.
\end{proof}

Since the Bayesian estimate of the random mean of $G$, i.e. $\E_{BS;n}\left[\int_0^{\infty} t dG(t)\right]$, is used to define estimators of several queueing characteristics, its posterior consistency is examined next. For the most prominent prior process, namely the Dirichlet process a rather general result is known. This is reviewed briefly. So, let $P\sim \mathcal{D}_{\alpha}$ be a random probability measure that is distributed according to a Dirichlet prior with finite measure $\alpha$ as parameter. Let $f:\mathcal{S}\rightarrow \R$ be measurable. Then $[$c.f. \citet{feigin1989linear}$]$ $\int \left|f\right| d\alpha <\infty \Rightarrow \int\left|f\right| dP <\infty $ with $\mathcal{D}_{\alpha}$ probability one and $\E_{\mathcal{D}_{\alpha}}\left[\int f dP\right]=\int f d\E_{\mathcal{D}_{\alpha}}\left[P\right]=\int f d\alpha$. This fact in combination with an assumption that the state space is countable makes it easy to show that posterior consistency of the random measure $P$ induces posterior consistency of the random mean of $P$. However, here neither the state space is assumed to be countable nor the random measure is assumed to possess a Dirichlet prior. Thus the posterior consistency of the random mean of $G$ which is drawn according to a beta-Stacy process is studied in more depth and an affirmative result is given.

\begin{lemma}\thlabel{lemma5}
Suppose $G(\cdot)$ is drawn according to a beta-Stacy process with parameters $(\alpha,\beta)$, where $\alpha$ is a measure on $\R_+$ and $\beta(s)\geq 1$ such that
\begin{enumerate}
\item[(1.)] $\int_{\R_+}[\beta(s)]^{-1}\alpha(ds)=\infty$
\item[(2.)] $\int_{\R_+}\exp\left[-\int_0^t [\beta(s)]^{-1}\alpha(ds)\right]dt<\infty$.
\end{enumerate}
Then the posterior expectation of the random mean of $G$ converges $G_0^{\infty}$-a.s.~to the true one. That is for $G_0^{\infty}$-almost all sequences of data $X_1^{\infty}$ it holds
\begin{align*}
 \E_{BS;n}\left[\int_{\R_+} t dG(t)\right]:= \E_{BS}\left[\int_{\R_+} t dG(t)\bigg\vert X_1^n\right]\stackrel{n\rightarrow \infty}{\longrightarrow}\int_{\R_+}tdG_0(t).
\end{align*}
\end{lemma}

\begin{proof}
First of all note that, due to \citet{epifani2003exponential}, the first moment of the random mean under the prior does exist and is given as
\begin{align*}
&\E_{BS;n}\left[\int_{\R_+} t dG(t)\right]= \int\limits_0^{\infty} \exp\left[-\int_0^t \frac{\alpha(ds)}{\beta(s)+M_n(s)}\right]\times
\left[\prod\limits_{\{i:X_i\leq t\}} \frac{\beta(X_i)+M_n(X_i)-1}{\beta(X_i)+M_n(X_i)}\right] dt,
\end{align*}
where $M_n(s)=\sum_{j=1}^n \delta_{X_j}[s,\infty)$. Thus, letting $N_n(s)=\sum_{i=1}^n \delta_{X_i}[0,s)$, one obtains
\begin{align*}
&\E_{BS;n}\left[\int_{\R_+} t dG(t)\right]\\
&=\int\limits_0^{\infty} \exp\left[-\int_0^t \frac{\alpha(ds)}{\beta(s)+M_n(s)}\right]
\exp\left[-\int_0^t \log\left(\frac{\beta(s)+M_n(s)}{\beta(s)+M_n(s)-1}\right)dN_n(s)\right] dt.
\end{align*}
Since $\beta(s)\geq 1$, by elementary properties of the logarithm, it follows
\begin{align*}
\frac{1}{\beta(s)+M_n(s)}\leq \log\left(\frac{\beta(s)+M_n(s)}{\beta(s)+M_n(s)-1}\right)\leq\frac{1}{\beta(s)+M_n(s)-1},
\end{align*}
which in turn implies
\begin{align*}
&\int\limits_0^{\infty} \exp\left[-\int_0^t \frac{\alpha(ds)}{\beta(s)+M_n(s)}\right]
\times  \exp\left[-\int_0^t \frac{N_n(ds)}{\beta(s)+M_n(s)-1}\right] dt\\
&\leq\E_{BS;n}\left[\int_{\R_+} t dG(t)\right]\\
&\leq \int\limits_0^{\infty} \exp\left[-\int_0^t \frac{\alpha(ds)+N_n(ds)}{\beta(s)+M_n(s)}\right]dt.
\end{align*}
such that it remains to show that the bounding terms converge to the mean of the true c.d.f..
From a straight-forward application of the monotone convergence theorem and by continuity of the exponential function, it follows
\begin{align*}
&\lim\limits_{n\rightarrow \infty} \E_{BS;n}\left[\int_{\R_+} t dG(t)\right]=\int\limits_0^{\infty} \exp\left[-\int_0^t \lim\limits_{n\rightarrow \infty} \frac{1/n N_n(ds)}{1/n \beta(s)+1/n M_n(s)}  \right] dt\\
&=\int\limits_0^{\infty} \exp\left[-\int_0^t  \frac{G_0(ds)}{1-G_0(s)}  \right] dt.
\end{align*}
Now, note that $\frac{G_0(ds)}{1-G_0(s)}=\lambda(s) ds=\bP(s< S \leq s+\Delta s | S>s)$, where $\lambda(s)=\lim_{\Delta s \rightarrow 0}\frac{\bP(s < S \leq s+\Delta s)}{\Delta s [1-G_0(s)]}$ usually denotes the hazard function. Hence $\Lambda(t)=\int_0^t \lambda(s) ds=\int_0^t \frac{G_0(ds)}{1-G_0(s)}$ is the cumulative hazard  until $t\geq 0$.  On the other hand $\Lambda(t)=-\log(1-G_0(t))$. Thus,
\begin{align*}
\exp\left[-\int_0^t\frac{G_0(ds)}{1-G_0(s)}\right]=\exp\left[-\int_0^t \lambda(s) ds\right]=\exp\left[-\Lambda(t)\right]=\exp\left[\log(1-G_0(t))\right]=1-G_0(t),
\end{align*}
which in turn completes the proof.
\end{proof}

The next lemma establishes the contraction of the mass of the posterior law of the random mean.

\begin{lemma}\thlabel{lemma6}
If, in addition to the assumptions of the previous lemma, the parameter $(\alpha,\beta)$ meets the condition
\begin{align*}
\int_{\R_+} \exp\left[-\int_0^{\sqrt{t}} \left[\beta(s)\right]^{-1}\alpha(ds)\right]dt<\infty
\end{align*}
then the posterior variance of the random mean $\int t G(dt)$ vanishes  $G_0^{\infty}$-a.s.~as the size of the data increases. That is, for $G_0^{\infty}$-almost all sequences of data $X_1^{\infty}$ it holds
\begin{align*}
\mathbb{V}_{BS;n}\left[\int_{\R_+} t G(dt)\right]:= \mathbb{V}_{BS}\left[\int_{\R_+} t G(dt) \bigg\vert X_1^n \right]\stackrel{n \rightarrow \infty}{\longrightarrow} 0.
\end{align*}
\end{lemma}
\begin{proof}
Due to \thref{lemma5} and the continuous mapping theorem the claim follows if the second moment of the random mean under the posterior converges a.s. to the square of the mean of the true distribution function $G_0(\cdot)$.
From \citet{epifani2003exponential} the second moment under the prior law exists and the second moment under the posterior law of the random mean is given by
\begin{align*}
&\E_{BS;n}\left[\left(\int_0^{\infty} t dG(t)\right)^2\right]\\
&=2 \int\limits_0^{\infty} \int\limits_r^{\infty} \exp\left[-\int_0^r \frac{\alpha(dx)}{\beta(x)+M_n(x)+1}\right] \exp\left[-\int_0^s \frac{\alpha(dx)}{\beta(x)+M_n(x)}\right]\\
&\qquad \qquad \qquad \qquad  \times \left(\prod\limits_{\{i:X_i\leq r\}}\frac{\beta(X_i)+M_n(X_i)}{\beta(X_i)+M_n(X_i)+1}\right)
\left(\prod\limits_{\{j:X_j\leq s\}}\frac{\beta(X_j)+M_n(X_j)-1}{\beta(X_j)+M_n(X_j)}\right) ds dr\\
&=2 \int\limits_0^{\infty} \exp\left[-\int_0^r \frac{\alpha(dx)}{\beta(x)+M_n(x)+1}\right] \exp\left[-\int_0^r \log\left(\frac{\beta(x)+M_n(x)}{\beta(x)+M_n(x)+1}\right)dN_n(x)\right]\\
&\qquad \qquad \qquad \int\limits_r^{\infty} \exp\left[-\int_0^s \frac{\alpha(dy)}{\beta(y)+M_n(y)}\right] \exp\left[-\int_0^s \log\left(\frac{\beta(y)+M_n(y)}{\beta(y)+M_n(y)-1}\right)dN_n(y) \right]ds dr\\
\end{align*}

Hence, by similar arguments as in the proof of the previous lemma and using Fubini's theorem, it follows that for $G_0^{\infty}-$almost all sequences of data

\begin{align*}
\E_{BS;n}\left[\left(\int_0^{\infty} t dG(t)\right)^2\right]\stackrel{n \rightarrow \infty}{\longrightarrow} &
2\int\limits_0^{\infty} \exp\left[-\int_0^r\frac{dG_0(x)}{1-G_0(x)}\right]\int\limits_r^{\infty} \exp\left[-\int_0^s \frac{dG_0(y)}{1-G_0(y)}\right] ds dr\\
& =2\int\limits_0^{\infty} \int\limits_r^{\infty} \exp\left[-\int_0^r\frac{dG_0(x)}{1-G_0(x)}\right] \exp\left[-\int_0^s \frac{dG_0(y)}{1-G_0(y)}\right] ds dr. \tag{\#}
\end{align*}

Furthermore, a straight-forward application of Fubini's theorem yields

\begin{align*}
\int\limits_0^{\infty}\int\limits_0^sg(r,s)dr ds=\int\limits_0^{\infty}\int\limits_r^{\infty} g(r,s) dsdr,
\end{align*}
where g is an arbitrary integrable function $g:[0,\infty)^2\rightarrow \R_+$.
Now, let $f(r,s)$ denote the integrand of the last integral of equation (\#). Since $f$ is symmetric, i.e. $f(r,s)=f(s,r)$, it follows from above equality and further application of Fubini's theorem
\begin{align*}
&2 \int\limits_0^{\infty} \int\limits_r^{\infty} f(r,s) ds dr=\int\limits_0^{\infty} \int\limits_0^{s} f(r,s) dr ds +\int\limits_0^{\infty} \int\limits_r^{\infty}f(r,s) ds dr \\
&=\int\limits_0^{\infty} \int\limits_0^{r}f(s,r)dsdr+\int\limits_0^{\infty} \int\limits_r^{\infty}f(r,s)dsdr=\int\limits_0^{\infty} \int\limits_0^{\infty} f(r,s) dr ds .
\end{align*}

Therefore,

\begin{align*}
\E_{BS;n}\left[\left(\int_0^{\infty} t dG(t)\right)^2\right]\stackrel{n \rightarrow \infty}{\longrightarrow}  \left(\int_0^{\infty} t dG_0(t)\right)^2,
\end{align*}

which completes the proof of the assertion.
\end{proof}

\begin{theorem}\thlabel{theorem7}
Under the assumptions of \thref{lemma5} and \thref{lemma6} the mean of the random c.d.f. $G$ possesses the property of posterior consistency. That is for all $\epsilon>0$ and $G_0^{\infty}$-almost all sequences of data it holds
\begin{align*}
\Pi_{BS}\left(\left|\int_{\R_+} t G(dt)-\int_{\R_+}tG_0(dt)\right|>\epsilon \bigg\vert X_1^n\right)\stackrel{n \rightarrow \infty}{\longrightarrow} 0.
\end{align*}
\end{theorem}
\begin{proof}
This is a direct consequence of the two previous lemmas and the Markov inequality.
\end{proof}

\begin{remark}
A brief discussion of the assumptions of above lemmas and theorem, respectively, is given. Assumption $(1.)$ of Lemma \ref{lemma5} is an artifact of survival analysis and ensures that the prior $(\alpha, \beta)$, that is to be chosen, actually leads to a cumulative hazard function.
Assumption (2.) of the first lemma and the additional assumption of the second lemma, respectively, ensures the existence of the first, resp. second, moment of the posterior distribution of the random mean. These conditions are given in \citet{epifani2003exponential} Proposition 4. 

Moreover, this work extends results which concern the existence of certain random functionals w.r.t. to a random c.d.f. drawn according to a NTR prior. Those results can be understood as a generalization of a work by \citet{feigin1989linear} which gives conditions under which certain functionals of a random measure drawn according to a Dirichlet prior exist in terms of the base measure (the prior parameter). At this place it should be emphasized that they investigated this problem by creating a new approach to the Dirichlet process which is often not mentioned in the literature. To be more precise, they show that the Dirichlet process can be extracted as the invariant distribution of a measure-valued Markov chain and exploit this theory to show sufficiency of the conditions of their existence theorem.
Since Dirichlet priors are as well included in the family of NTR priors the extension by Epifani et al. seems natural.
 
 The last assumption, i.e. that $\beta(s)\geq 1$ ensures posterior consistency of the random mean of the c.d.f. governed by the beta-Stacy prior and is used in the proofs of both lemmas. Hence, a prior for the random c.d.f. $G$ is to be chosen, in terms of its parameter, such that the sequence of posteriors of the $G$-mean is consistent as long as this property is desired. It is interesting to see how advanced information on the posterior consistency of the random mean influences the prior knowledge of the random c.d.f. itself. To put it another way, information concerning the posterior consistency of the random mean is indeed prior information on $G$. However, the other way around is not true in general, i.e. posterior consistency of the random c.d.f. does not generally imply posterior consistency of the random mean. That is because this functional is not continuous. However, posterior consistency of $G$ implies that of its truncated mean. That means integration is restricted to a compact set and would lead to a weaker form of consistency in the sense of compact convergence.
\end{remark}

As an immediate consequence of the previous results one obtains the consistency of the traffic intensity of the $M/G/1$ queue.

\begin{corollary}\thlabel{cor8}
Let $\mu:=\int s G(ds)$, $\mu_0:=\int sG_0(ds)$ be the random and true mean of the service-time distribution and $\rho_0=\lambda_0\mu_0$ the true traffic intensity. Further, let $\Pi_{BS\otimes\Gamma}=\Pi_{BS}\bigotimes\Pi_{\Gamma}$ be the prior on $(\cP(\R_+),\R_+)$ that is formed by taking the product-measure of $\Pi_{BS}$ and $\Pi_{\Gamma}$, $\Pi_{BS\otimes \Gamma;n}$ the posterior, respectively, and
 define $\hat{\lambda}_n:=\E_{\lambda;n}[\lambda]:=\E_{\Gamma}[\lambda|A_1^n]$, $\hat{\mu}_n:=\E_{BS;n}\left[\int_0^{\infty} t dG(t)\right]:=\E_{BS}\left[\int_0^{\infty} t dG(t) | S_1^n\right]$ and $\hat{\rho}_n:=\E_{BS\otimes\Gamma;n}[\rho]:=\E_{BS\otimes\Gamma}[\rho|A_1^n,S_1^n]$.\\ Then, under the assumptions of \thref{prop1} and \thref{theorem7}, one has
\begin{enumerate}[(i)]
\item for all $\epsilon>0$
\begin{align*}
\Pi_{BS\otimes\Gamma}(|\rho-\rho_0|\geq \epsilon |S_1^n,A_1^n)\stackrel{n\rightarrow\infty}{\longrightarrow}0,
\end{align*}
for $ P_{\lambda_0}^{\infty} \otimes G_0^{\infty}$-almost all data sequences $(A_1^{\infty},S_1^{\infty})$,
\item
\begin{align*}
|\hat{\rho}_n-\rho_0|\stackrel{n\rightarrow\infty}{\longrightarrow}0,
\end{align*}
for $ P_{\lambda_0}^{\infty} \otimes G_0^{\infty}$-almost all data sequences $(A_1^{\infty},S_1^{\infty})$.
\end{enumerate}
\end{corollary}
\begin{proof}
(i) Recalling the posterior consistency of the arrival rate $\lambda$ in \thref{prop1} and the service rate $\mu$ in \thref{theorem7} one has
\begin{align*}
&\Pi_{BS\otimes\Gamma;n} (|\rho-\rho_0|\geq \epsilon ) =\Pi_{BS\otimes\Gamma;n}(|\mu\lambda-\mu \lambda_0+\mu\lambda_0-\lambda_0\mu_0|\geq \epsilon)\\
&\leq \Pi_{BS\otimes\Gamma;n}(|\mu\lambda-\mu \lambda_0|+|\mu\lambda_0-\lambda_0\mu_0|\geq \epsilon )\\
&\leq \Pi_{BS\otimes\Gamma;n}\left(| \mu\lambda-\mu \lambda_0|\geq \epsilon/2 \right)+  \Pi_{BS\otimes\Gamma;n}\left(|\mu-\mu_0|\geq \frac{\epsilon}{2\lambda_0}\right)\\
&\leq  \Pi_{BS\otimes\Gamma;n}\left((\mu_0+\delta)|\lambda- \lambda_0|\geq \epsilon/2, |\mu-\mu_0|<\delta \right) \\
&\qquad +\Pi_{BS;n}\left( |\mu-\mu_0| \geq \delta \right)+ \Pi_{BS;n}\left(|\mu-\mu_0|\geq \frac{\epsilon}{2 \lambda_0}\right)\\
&\leq \Pi_{\Gamma;n}\left(|\lambda-\lambda_0|\geq \frac{\epsilon}{2 (\mu_0+\delta)}\right)+\Pi_{BS;n}\left( |\mu-\mu_0| \geq \delta \right)+\Pi_{BS;n}\left(|\mu-\mu_0|\geq \frac{\epsilon}{2\lambda_0} \right),
\end{align*}
from which the assertion of (i) follows.\\

(ii) Straightforward one has a.s.
\begin{align*}
|\hat{\rho}_n-\rho_0|&\leq |\hat{\mu}_n\lambda_0-\hat{\mu}_n\hat{\lambda}_n|+\lambda_0|\hat{\mu}_n-\mu_0|\\
& \leq(\mu_0+\gamma)|\lambda_0-\hat{\lambda}_n|+\lambda_0|\hat{\mu}_n-\mu_0|
\end{align*}
by selecting $\gamma>0$ such that a.s. $|\hat{\mu}_n-\mu_0|<\gamma$, for all $n$ sufficiently large.  The assertion of (ii) then follows from the proof of \thref{prop1} and \thref{lemma5} completing the proof.
\end{proof}

We are now in a position to state the main theorem of this section, i.e. the consistency of the estimators for the waiting time LST, the queue length p.g.f. and the system size p.g.f.,  defined in section 3.3.
%\begin{align*}
%&w_n^*(z)=\frac{z(1-\hat{\rho}_n)}{z-\hat{\lambda}_n(1-g_n^*(z))},\\
%&q_n^*(z)=\frac{(1-z)(1-\hat{\rho}_n)}{g_n^*(\hat{\lambda}_n(1-z))-z},\\
%&n_n^*(z)=g_n^*(\hat{\lambda}_n(1-z))q_n^*(z),
%\end{align*}
%respectively.

\begin{theorem}
Under the assumptions of \thref{prop1}, \thref{prop3} and \thref{theorem7} one has
\begin{enumerate}[(i)]
\item for all $\epsilon>0$
\begin{align*}
\Pi_{BS\otimes\Gamma}\left(\sup\limits_{0\leq z <\infty} |f(z)-f_0(z)|\geq\epsilon |S_1^n,A_1^n\right)\stackrel{n\rightarrow\infty}{\longrightarrow}0.
\end{align*}

\item
\begin{align*}
\sup\limits_{0\leq z <\infty} |f_n^*(z)-f_0(z)|\stackrel{n\rightarrow\infty}{\longrightarrow}0,
\end{align*}
for $ P_{\lambda_0}^{\infty} \otimes G_0^{\infty}$-almost all data sequences $(A_1^{\infty},S_1^{\infty})$, where $f(\cdot)\in\{n(\cdot),q(\cdot),w(\cdot)\}$.
\end{enumerate}
\end{theorem}
\begin{proof}

We show the result for $f=w$, that is in the case of the p.g.f. of the waiting-time distribution. The other cases are treated similarly and we omit the details. We intent to apply the continuous mapping theorem. Therefor note that the mapping $(g(\cdot), \lambda, \mu)\mapsto w(\cdot)$ is continuous with respect to the suitable topologies induced by the sup-norm. Hence, the assertion of the theorem is implied by the results of the present section.
\end{proof}

We shall stress that above theorem is important, especially for applications, since it states that one can reduce the difficult task of making Bayesian inference for queueing characteristics to that of making inference for the observables separately in order to compose it subsequently to that of the objects of interest.
Qualitatively this will lead to reasonable asymptotic inference.
This is appreciable because it is a non-feasible task to place a prior on the distributions of aforementioned characteristics in a way that it can be updated with data given by observations that we have access to.

\section{Posterior Normality}
Asymptotic normality results, so called Bernstein-von Mises theorems, for the posterior law serve as an additional validation of Bayesian procedures. Especially in situations where the exact posterior law is not available or hard to compute they are useful from a rather applied viewpoint to get an approximation of the posterior law. While the previous section has shown that the posterior concentrates around the true value when data increases, normality results give an idea how, asymptotically speaking, it does concentrate and how fluctuations around this centering appear. Often it is true that the posterior law, if centered and rescaled appropriately, resembles a centered Gaussian distribution with a certain covariance structure. For applications, special interest lies in this limiting covariance structure.\\

 The earliest result dates back to \citet{de1774memoire} who approximates the posterior of a beta distribution by a normal integral. Sergei Bernstein and Richard von Mises gave a rather modern version of this approach by applying a general result about infinite products of functions to the sequence of posterior densities, see e.g. \citet{johnson1967asymptotic} for a review and references. These approaches of expanding posterior densities coined the name. Nowadays, the parametric case is well understood and results including centering with the MLE or the Bayes estimate as well can be found in \citet{schervish1995theory} or \citet{ghosh2003bayesnonp}. However, the results needed for an asymptotic behavior of the posterior of infinite-dimensional parameters, i.e. in the nonparametric case, go deeper and deserve separate investigation. Posterior normality results involving Dirichlet process priors or NTR processes can be found in \citet{conti1999larges} who employs a result by \citet{freedman1963asymptotic} on asymptotic normality in the finite-dimensional case.\\

 We use a general result by \citet{kim2004bernst} on the asymptotic normality of NTR processes  to obtain the asymptotic behavior of the posterior law of the waiting time LST. This result is stated next. For the sake of clarity, it is stated for the non-censored case. However, it can be enlarged to the situation where data is right-censored. For a positive real number $\tau$ let $(D[0,\tau],\left\|\cdot\right\|_{\tau})$ denote the space of all cadlag functions on $[0,\tau]$ equipped with the sup-norm and $\cL[X]$ the law of a random object $X$.

\begin{theorem}[\citet{kim2004bernst}]\thlabel{kimlee}
Suppose that $G(\cdot)=1-\exp[A(\cdot)]$ is a random c.d.f. drawn according to a NTR prior $\Pi$ with corresponding increasing additive process $A(\cdot)$ whose Lévy measure $L$ is given by
\begin{align*}
L([0,t],B)=\int\limits_0^t \int\limits_{B} \frac{g_s(x)}{x}dx \lambda(s)ds,
\end{align*}
where $\int_0^1 g_t(x)dx=1$ for all $t\in\R_+$. Let the true c.d.f. $G_0$ be continuous and such that $G_0(t)<1$ for all $t\in\R_+$.
Moreover let the following conditions be fulfilled for a $\tau >0,$
\begin{itemize}
\item $\sup\limits_{t\in[0,\tau],x\in[0,1]} (1-x)g_t(x) <\infty$,
\item there is a function $q(t):\R_+ \rightarrow \R$ such that $0<\inf\limits_{t\in [0, \tau]} q(t) < \sup\limits_{t\in [0, \tau]}q(t) <\infty$ and, for some $\alpha>1/2$ and $\epsilon>0$,
\begin{align*}
\sup\limits_{t\in[0,\tau],x\in[0,\epsilon]}\left|\frac{g_t(x)-q(t)}{x^{\alpha}}\right|<\infty,
\end{align*}
\item $\lambda(t)$ is bounded and positive on $(0,\infty)$.
\end{itemize}
Then it holds that for $G_0^{\infty}$-almost all sequences of data $S_1^{\infty}$
\begin{align*}
\lim\limits_{n\rightarrow\infty}\Pi\left(\sqrt{n}[G(\cdot)- \E_{\Pi;n}[G(\cdot)]] \big\vert S_1^n \right) =\cL[\mathcal{A}(\cdot)]
\end{align*}
weakly on $(D[0,\tau],\left\|\cdot\right\|_{\tau})$, where $\mathcal{A}(\cdot)$ denotes a centered Gaussian process given by
$\mathcal{A}(t) = (G_0(t)-1) W(A_0(t))$ with a Brownian motion $W.$
The covariance structure then is  $h(u,v):=\textsf{Cov}[\mathcal{A}(u),\mathcal{A}(v)]=(1-G_0(u))(1-G_0(v))min(A_0(u),A_0(v))$.
\end{theorem}
\begin{proof}
See \citet{kim2004bernst} for the proof as well as for a discussion of the constraints of the theorem.
\end{proof}

Above theorem can well be stated as follows. The posterior distribution of the scaled and centered random process looks more and more like a Gaussian process as the sample size increases. The limiting process is centered, i.e. the expectation function is the constant function taking only the value zero.
Next result ensures that above assertions hold for beta-Stacy processes.

\begin{corollary}[\citet{kim2004bernst}]\thlabel{kimleecor}
The assertion of \thref{kimlee} for the random c.d.f. $G$ holds, with $G$ being governed by a beta-Stacy process prior with parameters $(c(\cdot),H(\cdot))$, where $H$ is a continuous distribution function with continuous density $a$ with respect to the Lebesgue measure such that $0<\inf_t c(t)[1-H(t)] <  \sup_t c(t)[1-H(t)]<\infty$, i.e. for $G_0^{\infty}$-almost all sequences of data $S_1^{\infty}$ it holds
\begin{align*}
\lim\limits_{n\rightarrow \infty} \Pi_{BS;n}(\sqrt{n}[G(\cdot)-\hat{G}_n(\cdot)])=\cL[\mathcal{A}(\cdot)]
\end{align*}
weakly on $(D^{(1)}[0,\infty),\left\|\cdot\right\|_{\infty})$ with $\mathcal{A}(\cdot)$ as in \thref{kimlee}, where $D^{(1)}[0,\infty)$ denotes the space of cadlag functions bounded by one on $[0, \infty)$.
\end{corollary}
\begin{proof}
By Theorem 5 of \citet{dey2003some}, a beta-Stacy process is a transformed beta process. More precisely, a process $\Lambda$ is a beta-Stacy process with parameters $(c(\cdot),H(\cdot))$ if and only if it is a beta process with parameters $\left(c(\cdot)[1-H(\cdot)],\int_0^{\cdot}\frac{dH(s)}{1-H(s)}\right)$. Further on, \citet{kim2004bernst} showed that a beta process fulfills the constraints of the theorem as long as $\frac{a(s)}{1-H(s)}$ is positive and continuous on $\R_+$ and $c(t)[1-H(t)]$ is as required.
\end{proof}

Next, we turn to the asymptotic behavior of the LST $g$ of the service time distribution $G$ and its random mean $\mu:=\int x dG(x)$. Since a plug-in estimator for $g$ is employed, we will use the functional delta method to provide the asymptotic normality result. Recall that in the previous section a posterior consistency result was given for the random mean $\mu$, i.e. it was shown that the posterior law of $\mu$ centers a.s.~around its true value $\mu_0:=\int x G_0(x)$. This suggests the conjecture that a normality result might be present in this case as well. However, even if \citet{regazzini2003distributional} provide results to approximate the density of the random mean of random measures, we were not able to show the density in our case to be approximated by a normal density if the sample size increases. The main reason for this is that it does not seem to be straight-forward to obtain a suitable expansion of this approximated density.

Hence, in the following we make the rather technical assumption that there is prior knowledge of the kind that service times can not exceed a certain sufficiently large threshold $M\in \R_+$. From a practical point of view this is a rather gentle constraint. Let $\F_M:=\F_M(\R_+):=\{F\in\F(\R_+) : F(t)=1, \forall t\geq M\}$ be the space of all c.d.f.'s whose corresponding probability measure has support  $[0,M]$. Since it is well known that the prior guess on the c.d.f. $G$ under a beta-Stacy prior $BS(c,H)$ is given as $\E_{BS;n}[G]=H$, in the following we take $H\in \F_M$ such that $H$ is continuous on $[0,M]$. Recall that $G(\cdot)=1-\exp[-A(\cdot)]$, where $A$ is a non-negative increasing additive process with Lévy measure
\begin{align*}
dN_t(x)=\frac{dx}{1-e^{-x}} \int\limits_0^{t} e^{-xc(s)[1-H(s)]} c(s) H(ds).
\end{align*}
From the theory of increasing additive processes it is well known $[$see e.g. \citet{sato1999levy}$]$ that for all $t>0$ $A(t)=A_t$ is a random variable governed by a infinitely divisible distribution $\phi_t$. Let $\hat{\phi}_t(\xi)$ denote the characteristic function of $\phi$, i.e.
\begin{align*}
\hat{\phi}_t(\xi)=\int \exp[i\xi s] \phi(ds)= \exp\left[-\int_0^{\infty} (1-e^{i\xi x})dN_t(x)\right].
\end{align*}
Furthermore, since $(A_t)$ has independent increments one has $\hat{\phi}_t(\xi)=\hat{\phi}_s(\xi)\hat{\phi}_{s,t}(\xi)$ for all $s<t$, where $\hat{\phi}_{s,t}(\xi)$ denotes the characteristic function of the difference $A_t-A_s$. Now, since $H\in\F_M$ it follows for $t>M,$

\begin{align*}
&\hat{\phi}_{t}(\xi)=\exp\left[-\int_0^{\infty} \left(1-e^{i\xi x}\right)dN_t(x)\right]
=\exp\left[-\int_0^{\infty} \frac{1-e^{i\xi x}}{1-e^{-x}} \int_0^t e^{-xc(s)[1-H(s)]}c(s)  H(ds)  dx \right] \\
& =\exp\left[-\int_0^{\infty} \frac{1-e^{i\xi x}}{1-e^{-x}} \int_0^M e^{-xc(s)[1-H(s)]}c(s)   H(ds) dx \right] \\& \qquad \qquad \cdot \exp\left[-\int_0^{\infty} \frac{1-e^{i\xi x}}{1-e^{-x}} \int_M^t e^{-xc(s)[1-H(s)]}c(s) H(ds) dx  \right]\\
&=\hat{\phi}_M(\xi),
\end{align*}
the process $(A_t)_{t\geq M}$ is a.s. constant. That implies that the corresponding c.d.f. $G$ is constant from $M$ onwards. In order to ensure that it is indeed a distribution function, we set $G(t):=1$, for all $t\geq M$. The support of the truncated prior $\Pi_{BS}^{(M)}$ law will still be all of $\F_M$.

In order to achieve a posterior normality result for the random LST $g$, we show that the mapping $\Phi: D^{(1)}[0,M]\rightarrow C^{(1)}[0,\infty); G\mapsto\int_0^{\infty}e^{-sz}dG(s)$ is Hadamard differentiable, where $D^{(1)}[0,M]$ and $C^{(1)}[0,\infty)$ denote the space of cadlag and continuous functions, bounded by one, respectively. Moreover, since we are solely interested in distribution functions on $\R_+$, w.l.o.g. it is assumed that $D^{(1)}[0,M]$ consists only of functions starting at zero. For a good reference on Hadamard differentiability and applications including the functional delta method see \citet{kosorok2008introduction}.\\

\begin{lemma}
The mapping
\begin{align*}
\Phi: (D^{(1)}[0,M],\left\|.\right\|_{M})&\rightarrow (C^{(1)}[0,\infty),\left\|.\right\|_{\infty})\\
 G&\mapsto \Phi[G](\bullet):=\int_0^{M}e^{-\bullet s}G(ds)
\end{align*}
is Hadamard differentiable.
\end{lemma}
\begin{proof}

Let $t\searrow 0$ and $h_t\in D^{(1)}[0,M]$, such that $h_t\stackrel{t\searrow0}{\longrightarrow}h \in D^{(1)}[0,M]$ w.r.t. the sup-norm. Then, by properties of the Riemann-Stieltjes integral, one has
\begin{align*}
&\left\|\frac{\Phi[G+th_t](z)-\Phi[G](z)}{t}-\int_0^{M}e^{-zs}dh(s)\right\|_{\infty}\\
&=\sup\limits_{0\leq z <\infty} \left| \int_0^{M}[h_t(s)-h(s)] (-z)e^{-zs} ds\right|\\
&\leq \sup\limits_{0\leq z <\infty}  \left| \int_0^{M} \left(\sup\limits_{0\leq x <M}|h_t(x)-h(x)|\right) (-z)e^{-zs} ds\right|\\
& \leq \left\|h_t-h\right\|_{M} \sup\limits_{0\leq z <\infty} \int_0^{M}ze^{-zs} ds\\
&=\left\|h_t-h\right\|_{M} \stackrel{t\searrow 0}{\longrightarrow} 0.
\end{align*}
Define
\begin{align*}
\Phi'_G:D^{(1)}[0,M]&\rightarrow C^{(1)}[0,\infty)\\
h&\mapsto\Phi'_G[h](z)=\int_0^{M}e^{-zs}dh(s).
\end{align*}
Since the Riemann-Stieltjes integral is linear in the integrator, the mapping $\Phi_G'$ is linear. Moreover, if $\sup\limits_{0\leq s <M}\left|F(s)-G(s)\right|<\delta$, it follows
\begin{align*}
\left|\int\limits_0^{M} e^{-zs} dF(s)-\int\limits_0^{M} e^{-zs} dG(s)\right|\leq \left\|F-G\right\|_{M} \sup\limits_{0\leq z<\infty} \int\limits_0^{M} ze^{-zs}ds <\delta,
\end{align*}
thus the continuity of $\Phi'_G.$
Hence the mapping $\Phi$ is Hadamard differentiable with derivative $\Phi'_G.$
\end{proof}

Now, the lemma will be applied in combination with the functional delta method to obtain the posterior normality of the service time LST centered suitably at its respective Bayes estimator. Write $\Pi_{BS;n}^{(M)}$ for the posterior law induced by the $M$-truncated beta-Stacy process.

\begin{corollary}\thlabel{corFasching}
Let $g_n^{*M}(z)=\int_0^{M} e^{-zs} d\E_{BS;n}[G](s)$. Under the assumptions of \thref{kimlee} and \thref{kimleecor}, it holds for $G_0^{\infty}$-almost all sequences $S_1^{\infty}$
\begin{align*}
\lim\limits_{n\rightarrow \infty} \Pi_{BS;n}^{(M)} \left(\sqrt{n}\left[ g(\cdot)-g_n^{*M}(\cdot) \right]\right) = \mathcal{L}[\G(\cdot)]
\end{align*}
on $C^{(1)}([0,\infty),||.||_{\infty} )$,
where $\G(z)$ is a centered Gaussian process with covariance structure $\gamma(\cdot,\cdot)$ given by
\begin{align*}
\gamma(u,v)=\textsf{Cov}\left[\G(u),\G(v)\right]=uv\int_0^{M}\int_0^{M}e^{-(us+vt)}h(u,v)  du dv,
\end{align*}
where $h(\cdot, \cdot)$ is defined in Theorem \ref{kimlee}.
\end{corollary}
\begin{proof}
By the functional delta method applied to the mapping in the previous lemma one has
\begin{align*}
\G(z)=\Phi_G'\left[(G_0(\cdot)-1)W(A_0(\cdot))\right](z)=\int_0^{M}e^{-zs}d\left[(G_0(s)-1)W(A_0(s)\right].
\end{align*}
Using a Riemann sum approximation for above integral, one concludes that the process $\G(\cdot)$ is a Gaussian process. Further, by well-known properties of the Riemann-Stieltjes integral, one gets
\begin{align*}
\G(z)=z\int_{0}^{M} \left(1-G_0(s))W(A_0(s)\right)  e^{-sz}  ds.
\end{align*}
Using Fubini's theorem it is immediately seen that $\E[\G(z)]=0$ for any $z\in\R_+$.
Again using Fubini's theorem,  the covariance structure of $\G(\cdot)$ is obtained as
\begin{align*}
&\textsf{Cov}\left[\G(u),\G(v)\right]=E\left[\G(u)\G(v)\right]\\
&=uv\E\left[\int_0^{M}\int_0^{M}(1-G_0(s))W(A_0(s))  e^{-us}(1-G_0(t))W(A_0(t))  e^{-vt} ds dt\right]\\
&=uv\int_0^{M}\int_0^{M}e^{-(us+vt)}(1-G_0(s))(1-G_0(t)) \E\left[W(A_0(s))W(A_0(t))\right]  ds dt\\
&=uv\int_0^{M}\int_0^{M}e^{-(us+vt)}h(s,t)  ds dt.
\end{align*}
\end{proof}

Next, we investigate the posterior normality of the mean of the random c.d.f. $G$. Since no exact results seem obtainable, we use the plug-in estimator $\mu_n^*:=\int_0^M [1-\E_{BS;n}[G(t)]] dt$ which, in general, does not equal $\E_{BS;n}\left[\int_0^M [1-G(t)] dt\right]$. This estimator in combination with the $M$-truncated c.d.f.'s enables us to use the functional delta method for obtaining normality results.

\begin{lemma}
Let $M$ be an arbitrary positive real number. Then, the mapping
\begin{align*}
\Psi:(\F_M,\left\|\cdot\right\|_{M}) &\rightarrow ([0,M],\left|\cdot\right|)\\
            &G \mapsto \Psi[G]:=\int_0^{M}[1-G(s)] ds
\end{align*}
is Hadamard-differentiable with derivative $\Psi'[h]= - \int_0^M  h(s) ds$.
\end{lemma}
\begin{proof}
Take $t\searrow 0$ and $h_t\in \F_M$, such that $h_t\stackrel{t\searrow0}{\longrightarrow}h \in\F_M$. Then
\begin{align*}
&\left|\frac{\Psi[G+th_t]-\Psi[G]}{t}+\int_0^{M} h(s) ds\right|=\left|\int_0^M h(s)-h_t(s) ds \right|\leq M \left\|h-h_t\right\|_M \stackrel{t\searrow 0}{\longrightarrow}0.
\end{align*}
Obviously, the derivative of $\Psi$ is linear and continuous w.r.t. to the considered topologies.
\end{proof}

\begin{corollary}
Under the assumptions of \thref{kimlee} and \thref{kimleecor} it holds for $G_0^{\infty}$ almost all data $S_1^{\infty}$ that
\begin{align*}
\lim\limits_{n\rightarrow\infty} \Pi_{BS;n}^{(M)}\left(\sqrt{n}[\mu-\mu_n^*]\right)=\cL[\mathcal{H}],
\end{align*}
where $\mathcal{H}$ is a centered Gaussian random variable with variance $\eta:=\textsf{Var}[\mathcal{H}]=\int_0^M\int_0^M  h(s,t) dsdt$.
\end{corollary}
\begin{proof}
By the previous lemma and the functional delta method  the limiting variable is given by
$\Psi'[(G_0(\cdot)-1)W(A_0(\cdot))]=\int_0^M (1-G_0(s)) W(A_0(s)) ds$ which is seen to be centered Gaussian by a Riemann sum approximation in combination with Fubini's theorem.
Moreover, again by Fubini' theorem
\begin{align*}
&\textsf{Var}[\mathcal{H}]=\E\left[\mathcal{H}^2\right]=\E\left[\int_0^M\int_0^M (1-G_0(s))W(A_0(s))(1-G_0(t))W(A_0(t)) dsdt\right]\\
&=\int_0^M\int_0^M h(s,t) dsdt.
\end{align*}
\end{proof}

Next we consider the  asymptotic normality of the arrival-rate when centering with its Bayes estimate.\\

\begin{proposition}
Let $\hat{\lambda}_n:=\E_{\Gamma;n}\left[\lambda \right]$. Then, for $P_{\lambda_0}^{\infty}$-almost all sequences of data $A_1^{\infty}$ one has
\begin{align*}
\lim\limits_{n\rightarrow \infty}\Pi_{\Gamma}\left(\sqrt{n}\left[\lambda-\hat{\lambda}_n\right] \big\vert A_1^n \right)=\mathcal{N},
\end{align*}
where $\mathcal{N}$ is a centered Gaussian random variable with precision $\lambda_0^2$.
\end{proposition}
\begin{proof}
By Theorem 1.4.3. in \citet{ghosh2003bayesnonp} the convergence of the posterior distribution of $\sqrt{n}\left[\lambda-\hat{\lambda}_n\right]$ to a centered normal distribution directly follows. Moreover, the variance of this limiting Gaussian variable is given by the inverse Fisher information at the true arrival rate. Checking the necessary conditions for interchanging integral and derivative is left to the interested reader. The Fisher information is obtained as
\begin{align*}
&\mathcal{I}(\lambda_0)=\E_{\lambda_0}\left[\left(\frac{\partial}{\partial \lambda} \log\left(\lambda e^{-\lambda A}\right)\right)^2\right]=-\E_{\lambda_0}\left[\frac{\partial^2}{\partial \lambda^2}[\log(\lambda) -\lambda A]\right]_{|_{\lambda=\lambda_0}}=\lambda_0^{-2}.
\end{align*}
\end{proof}

We are now in a position to formulate the posterior normality of the waiting-time LST. However, the same techniques can be applied to show posterior normality of several other queueing characteristics like e.g.~for the queue length p.g.f.~or the sojourn time LST.
 The waiting-time distribution is of special interest since it gives a qualitative idea about the loss of information that can occur in a $M/G/1$ system and thus helps to ensure a well-working system.  However, since the exact posterior law of the waiting time distribution is not obtainable in closed form, asymptotic approximations are given. These results extend results of the previous section where it was shown that the plug-in estimator is reasonable to make inference. Roughly, we prove  that the posterior law of the LST follows, asymptotically speaking, a Gaussian quantity that is centered around the plug-in estimator. Let $\Pi_{BS\otimes \Gamma}=\Pi_{BS}^{(M)}\bigotimes\Pi_{\Gamma}$ denote the prior on the parameter space $\R_+\times \mathcal{F}_M.$

\begin{theorem}
Let $w(z)=\frac{z(1-\rho)}{z-\lambda(1-g(z))}$ be the LST of the waiting time distribution as given in section 2 and $w_n^{*M}(z)=\frac{z(1-\hat{\lambda}_n \mu_n^*)}{z-\hat{\lambda}_n(1-g_n^{*M}(z))}$ be its plug-in estimator. Then, under the assumptions of \thref{kimlee} and \thref{kimleecor}, for $G_0^{\infty}\bigotimes P_{\lambda_0}^{\infty}$-almost all sequences of data $(S,A)_1^{\infty}$ it holds that
\begin{align*}
\lim\limits_{n\rightarrow\infty}\Pi_{BS\otimes\Gamma}\left(\sqrt{n}\left[w(\cdot)-w_n^{*M}(\cdot)\right]|A_1^n, S_1^n \right)=\cL[\mathcal{Z}(\cdot)]
\end{align*}
weakly on $C([0,\infty),||.||_{\infty} )$,
where $\mathcal{Z}(\cdot)$ is a centered Gaussian process with covariance structure $\zeta(u,v)=\textsf{Cov}[\mathcal{Z}(u),\mathcal{Z}(v)]$ given by
\begin{align*}
\begin{split}
\zeta(u,v)&=\frac{w_0(u)w_0(v)}{(1-\rho_0)^2}\left[\lambda_0^2\eta+\frac{\mu_0^2}{\lambda_0^2}\right]
+\lambda_0^{-2}\frac{w_0(u)(1-g_0(u))}{\lambda_0(1-g_0(u))-u}\times\frac{w_0(v)(1-g_0(v))}{\lambda_0(1-g_0(v))-v}\\
&+ \frac{w_0(u)w_0(v) \lambda_0^2}{\left[\lambda_0(1-g_0(u))-u\right]\left[\lambda_0(1-g_0(v))-v\right]}\gamma(u,v)  \\
&- \frac{\mu_0 w_0(u) w_0(v)}{[\lambda_0(1-\rho_0)]^2 }\left[\frac{w_0(u) (1- g_0(u))}{u}+\frac{w_0(v) (1- g_0(v))}{v}\right]\\
&- \frac{\lambda_0w_0(u)w_0(v)}{(1-\rho_0)^2} \Bigg[ w_0(u)\int\limits_{[0,M]^2} e^{-us}h(s,t)d(s,t)
%&\qquad \qquad \qquad \qquad
+w_0(v)\int\limits_{[0,M]^2} e^{-tv}h(s,t)d(s,t) \Bigg],
\end{split}
\end{align*}
where $\gamma(\cdot, \cdot)$ is given in Corollary  \ref{corFasching}.
\end{theorem}
\begin{proof}
Taking a similar route of proving as in the proof of Theorem 3 in \citet{conti1999larges}, we begin with a decomposition of $\sqrt{n}\left[w(z)-w_n^{*M}(z)\right]$. The decomposition yields
\begin{align*}
\sqrt{n}\left[w(z)-w_n^{*M}(z)\right]=&\frac{-w(z)}{1-\rho}\sqrt{n}\left[\lambda\mu-\hat{\lambda}_n\mu_n^*\right]\\
&+z(1-\hat{\lambda}_n\mu_n^*)\sqrt{n}\left[\frac{1}{z-\lambda(1-g(z))}-\frac{1}{z-\hat{\lambda}_n(1-g(z))}\right]\\
&+z(1-\hat{\lambda}_n\mu_n^*)\sqrt{n}\left(\frac{1}{z-\hat{\lambda}_n(1-g(z))}-\frac{1}{z-\hat{\lambda}_n(1-g_n^{*M}(z))}\right)\\
&=:Z_{1;n}(z)+Z_{2;n}(z)+Z_{3;n}(z)
\end{align*}
Now, the three terms of the sum are investigated separately and it will be shown that they possess the same asymptotic distribution as objects whose asymptotic is easier to obtain. These objects will be tagged by an additional $^*$ superscript. First write
\begin{align*}
Z_{1;n}(z)=\frac{-w(z)}{1-\rho}\sqrt{n}\left[\mu_n^*(\lambda-\hat{\lambda}_n)+(\mu-\mu_n^*)(\lambda-\lambda_0)+\lambda_0(\mu-\mu_n^*)\right]
\end{align*}
and note that by the uniform posterior consistency results of the previous section and the continuity of the mapping $G\mapsto \int_0^M (1-G(s)) ds$ one has
\begin{align*}
&\lim\limits_{n\rightarrow\infty}\Pi_{BS\otimes\Gamma}\left[Z_{1;n}(z)\big\vert A_1^n,S_1^n\right]=\lim\limits_{n\rightarrow\infty}\Pi_{BS\otimes \Gamma}\left[\frac{-w_0(z)}{1-\rho_0} \sqrt{n}(\mu_0(\lambda-\hat{\lambda}_n)+\lambda_0(\mu-\mu_n^*))\bigg\vert A_1^n, S_1^n\right]\\
&=:\lim\limits_{n\rightarrow\infty} \Pi_{BS\otimes \Gamma}\left[Z_{1;n}^*(z)\big\vert A_1^n, S_1^n\right]=:\cL[\mathcal{Z}_1(z)].
\end{align*}

For $Z_{2;n}(z)$, note that the mapping $\lambda\mapsto\left[z-\lambda(1-g(z))\right]^{-1}$ is analytic in a suitably chosen neighborhood of $\lambda_0$. Its derivative is given by
\begin{align*}
\lambda\mapsto \frac{1-g(z)}{\left[z-\lambda(1-g(z))\right]^2}.
\end{align*}
Thus, a Taylor expansion of that mapping yields
\begin{align*}
Z_{2;n}(z)=(1-\hat{\lambda}_n\mu_n^*)z\frac{1-g(z)}{\left[z-\bar{\lambda}(1-g(z))\right]^2}\sqrt{n}\left(\lambda-\hat{\lambda}_n\right),
\end{align*}
for a suitably chosen $\bar{\lambda}\in[\lambda_0,\hat{\lambda}_n]$. Therefore, using consistency results and continuous mapping, one gets
\begin{align*}
&\lim\limits_{n\rightarrow\infty}\Pi_{BS\otimes\Gamma}\left[Z_{2;n}(z)\big\vert S_1^n, A_1^n\right]
=\lim\limits_{n\rightarrow\infty}\Pi_{\Gamma}\left[(1- \rho_0)z\frac{1-g_0(z)}{\left[z-\lambda_0(1-g_0(z))\right]^2}\sqrt{n}\left(\lambda-\hat{\lambda}_n\right)\bigg\vert  A_1^n\right]\\
&=:\lim\limits_{n\rightarrow\infty}\Pi_{\Gamma}\left[Z_{2;n}^* \big\vert A_1^n\right]=:\cL[\mathcal{Z}_2(z)].
\end{align*}

Another Taylor expansion for the mapping $x\mapsto \left[z-\hat{\lambda}_n(1-x)\right]^{-1}$ and analogous reasoning as before yields

\begin{align*}
&\lim\limits_{n\rightarrow\infty}\Pi_{BS\otimes\Gamma}\left[Z_{3;n}(z)\big\vert S_1^n, A_1^n\right]
=\lim\limits_{n\rightarrow\infty}\Pi_{BS}^{(M)}\left[- \frac{w_0(z) \lambda_0}{\lambda_0(1-g_0(z))-z} \sqrt{n}\left[g(z)-g_n^{*M}(z)\right] \bigg\vert S_1^n\right]\\
&=:\lim\limits_{n\rightarrow\infty}\Pi_{BS}^{(M)}\left[Z_{3;n}^* \big\vert S_1^n\right]=:\cL[\mathcal{Z}_3(z)].
\end{align*}

Now, the convergence of the posterior law of the waiting time LST follows from the previous results of the present section. What remains is the calculation of the covariance structure. This, in turn, is easily obtained by above decomposition and the assumed independence of the arrivals and services or their prior laws, respectively.
\begin{align*}
\textsf{Cov}\left[\sum\limits_{i=1}^3 \mathcal{Z}_i(u),\sum\limits_{i=1}^3 \mathcal{Z}_i(v)\right]&=\sum\limits_{i=1}^3\textsf{Cov}\left[\mathcal{Z}_i(u),\mathcal{Z}_i(v)\right]+\sum\limits_{i\neq j}\textsf{Cov}\left[\mathcal{Z}_i(u),\mathcal{Z}_j(v)\right]\\
&=\sum\limits_{i=1}^3\textsf{\textsf{Cov}}\left[\mathcal{Z}_i(u),\mathcal{Z}_i(v)\right]\\
&+\textsf{\textsf{Cov}}\left[\mathcal{Z}_1(u),\mathcal{Z}_2(v)\right]+\textsf{Cov}\left[\mathcal{Z}_2(u),\mathcal{Z}_1(v)\right]\\
&+\textsf{Cov}\left[\mathcal{Z}_1(u),\mathcal{Z}_3(v)\right]+\textsf{Cov}\left[\mathcal{Z}_3(u),\mathcal{Z}_1(v)\right]
\end{align*}
By the previous results of this section it follows
\begin{align*}
\sum\limits_{i=1}^3 \textsf{Cov}\left[\mathcal{Z}_i(u), \mathcal{Z}_i(v)\right]
&=\frac{w_0(u)w_0(v)}{(1-\rho_0)^2}\left[\lambda_0^2\eta+\frac{\mu_0^2}{\lambda_0^2}\right]\\
&+\lambda_0^{-2}\frac{w_0(u)(1-g_0(u))}{\lambda_0(1-g_0(u))-u}\times\frac{w_0(v)(1-g_0(v))}{\lambda_0(1-g_0(v))-v}\\
&+ \frac{w_0(u)w_0(v) \lambda_0^2}{\left[\lambda_0(1-g_0(u))-u\right]\left[\lambda_0(1-g_0(v))-v\right]}\gamma(u,v).
\end{align*}
Furthermore, by the independence assumption of the prior laws of the inter-arrival rate and the service time distribution, one has
\begin{align*}
\textsf{Cov}\left[\mathcal{Z}_1(u),\mathcal{Z}_2(v)\right]=\frac{\mu_0}{\lambda_0^2(1-\rho_0)}\times \frac{w_0(u)w_0(v) (g_0(v)-1)}{v-\lambda_0(1-g_0(v))}.
\end{align*}

Furthermore, using the previous results of the present section and Fubini's theorem, one has

\begin{align*}
&\textsf{Cov}\left[\mathcal{Z}_1(u),\mathcal{Z}_3(v)\right]=\frac{w_0(u)w_0(v)\lambda_0^2}{(1-\rho_0)(\lambda_0(1-g_0(v))-v)}
\E[\mathcal{H}\times\mathcal{G}(v)]\\
&=\frac{w_0(u)w_0(v)\lambda_0^2}{(1-\rho_0)(\lambda_0(1-g_0(v))-v)}\\
& \qquad \times \E\left[\int_0^{M} (1- G_0(s))W(A_0(s))ds \times v\int_0^{M}(G_0(t)-1)W(A_0(t))e^{-tv}dt \right]\\
&=\frac{vw_0(u)w_0(v)\lambda_0^2}{(1-\rho_0)(\lambda_0(1-g_0(v))-v)}\int_0^{M} \int_0^{M} e^{-tv}h(s,t) ds dt\\
&=-\frac{w_0(u)w_0(v)}{(1-\rho_0)^2} \lambda_0^2w_0(v)\int_{[0,M]^2} e^{-tv}h(s,t) d(s,t).
\end{align*}
Finally, compounding above covariance structures yields $\zeta(\cdot,\cdot)$.
\end{proof}

The covariance structure $\zeta(\cdot,\cdot)$ depends on the unknown objects. However, one can use the suggested estimators and plug them in place of the true ones. This might be helpful to implement the problem and the provided consistency results ensure accuracy as long as the sample size is large enough.

\textbf{Funding:} This work was supported by the Deutsche Forschungsgemeinschaft (German Research
Foundation) within the programme "Statistical Modeling of Complex Systems and
Processes---Advanced Nonparametric Approaches", grant GRK 1953.\\

\begin{figure}[htbp]
	\begin{minipage}{0.6\textwidth} 
	\textbf{Corresponding Author}\\
	Dr. Cornelia Wichelhaus\\
	Technische Universit{\"a}t Darmstadt\\
	Schlossgartenstra{\ss}e 7\\
  64289 Darmstadt\\
	Germany\\
	wichelhaus@mathematik.tu-darmstadt.de
	\end{minipage}
	\hfill
	\begin{minipage}{0.6\textwidth}
	Moritz von Rohrscheidt\\
	Ruprecht-Karls Universit{\"a}t Heidelberg\\
	Berliner Stra{\ss}e 41-49\\
  69120 Heidelberg\\
	Germany\\
	rohrscheidt@uni-heidelberg.de
	\end{minipage}
\end{figure}
\bibliographystyle{humannat}
\bibliography{mybib}

\end{document}